\documentclass[11pt]{article}
\usepackage[usenames]{color}

\usepackage{graphicx,amsmath,amssymb,cite}
\usepackage{epsfig}
\DeclareGraphicsExtensions{.eps}

\newcommand{\matr}[1]{\mathrm{#1}}  

\newcounter{vecdepth}
0
\newcommand{\vect}[1]{          
  \ifnum\thevecdepth=0%
    \addtocounter{vecdepth}{1}%
    \mathbf{#1}%
    \addtocounter{vecdepth}{-1}%
  \else%
    \underline{#1}%
  \fi%
}

\newcommand{\restrict}[2]{      
   {\leavevmode\lower-.2em\hbox{${{#1}_{/}}_{\! {#2}}$}}
}

\newcommand{\isdef}{            
   \setlength\unitlength{1em}
   =\put(-.975,.48){{\tiny $\triangle$}}
}

\newcommand{\auth}[1]{{\sc #1}}     

\newtheorem{remarkk}{Remark} 

\newtheorem{notee}{Note}

\newtheorem{openn}{Problem}

\newtheorem{definitionn}{Definition}

\newtheorem{examplee}{Example}

\def\probitem[#1]#2\par{\item[{\sf #1} \hspace{0ex}] #2\par}    

\newenvironment{naritemize}{%
  \begin{itemize}%
    \vspace{-\topsep}%
    \vspace{-\parskip}%
    \setlength{\itemsep}{0pt}%
    \setlength{\parskip}{0pt}%
}{%
  \end{itemize}%
}

\newcounter{seqenum}

\input xypic
\xyoption{all}

\def\lab#1{{\txt{\footnotesize $#1$}}}  

\def\node{*=0{\bullet}}         

\def\edge#1[#2]{\ar@/^#1/@{-}[#2]}  
\def\dedge#1[#2]{           
  \ar@/^#1/@{-}[#2]
}

\def\curledgeee[#1]{\ar@{.}[#1]}      

\def\curledge[#1]{\ar@{~}[#1]}      
\def\inedge<#1,#2>{         
   \save[]+<#1,#2>\edge{0cm}[]\restore
}
\def\ftxt<#1,#2>#3{         
  \save[]+<#1,#2>*\txt{#3}\restore
}
\def\emline#1#2#3#4#5#6{%
\put(#1,#2){\special{em:moveto}}%
\put(#4,#5){\special{em:lineto}}}

\def\diagEta{
  \xymatrix@R=.08cm@C=2.2cm{
    & \node
         \inedge<-.35cm,.3cm>\inedge<-.15cm,.3cm>\inedge<.35cm,.3cm>
         \ftxt<0cm,.5cm>{$\overbrace{\qquad}$}
         \ftxt<0cm,.8cm>{\footnotesize $R_{_{\{\widetilde{k},\widetilde{0}\}}}$}
         \ftxt<1.1em,-.02ex>{$ u_{_{0}}$} & \\
     \\    \\    \\
      & \ftxt<0cm,0cm>{$\vdots$} &
     \\    \\    \\    \\
 \ftxt<2.5cm,.5cm>{\footnotesize $R_{_{\{\widetilde{k},\widetilde{k-1}\}}}$}
    & \node
         \inedge<-.35cm,.3cm>\inedge<-.15cm,.3cm>\inedge<.35cm,.3cm>
         \ftxt<2em,.4ex>{$u_{_{{k-1}}}$} & \\
    \node\ftxt<-1.7ex,0ex>{$x$}
       \edge{0cm}[uuuuuuuuur]\edge{0cm}[ur] \edge{0cm}[dr]\edge{0cm}[dddddddddr]
      & & \node\ftxt<1.7ex,0ex>{$y$}
      \edge{0cm}[uuuuuuuuul]\edge{0cm}[ul]\edge{0cm}[dl]\edge{0cm}[dddddddddl] \\
            & \node
         \inedge<-.35cm,-.3cm>\inedge<.15cm,-.3cm>\inedge<.35cm,-.3cm>
         \ftxt<0cm,-.6cm>{\footnotesize $R_{_{\{\widetilde{k},\widetilde{k+1}\}}}$}
         \ftxt<2em,-.9ex>{$u_{_{{k+1}}}$} &
    \\   \\    \\    \\
     & \ftxt<0cm,0cm>{$\vdots$} &
    \\   \\    \\    \\
    & \node
         \inedge<-.35cm,-.3cm>\inedge<.15cm,-.3cm>\inedge<.35cm,-.3cm>
         \ftxt<0cm,-.5cm>{$\underbrace{\qquad}$}
         \ftxt<0cm,-.8cm>{\footnotesize $R_{_{\{\widetilde{k},\widetilde{n-1}\}}}$}
         \ftxt<1.6em,-.4ex>{$u_{_{{n-1}}}$} & \\
  }
}
\def\diagzeta{
  \xymatrix@R=.15cm@C=.8cm{
    & \node
         \inedge<-.35cm,.3cm>\inedge<-.15cm,.3cm>\inedge<.35cm,.3cm>
         \ftxt<0cm,.5cm>{$\overbrace{\qquad}$}
         \ftxt<0cm,.8cm>{\footnotesize $R_{_{\{\widetilde{k},\widetilde{i}\}}}$} & \\\ftxt<6.4ex,.4ex>{$v$}
         \\
          \node\ftxt<-1.7ex,0ex>{$a$} \edge{0cm}[uur]\edge{0cm}[ddr] &
      & \node\ftxt<1.7ex,0ex>{$b$} \edge{0cm}[uul]\edge{0cm}[ddl] \\\ftxt<6.4ex,-.5ex>{$w$}
      \\
      & \node
         \inedge<-.35cm,-.3cm>\inedge<.15cm,-.3cm>\inedge<.35cm,-.3cm>
         \ftxt<0cm,-.5cm>{$\underbrace{\qquad}$}
         \ftxt<0cm,-.8cm>{\footnotesize $R_{_{\{\widetilde{k},\widetilde{j}\}}}$} & \\
  }
}
\def\diagtree{
  \xymatrix@R=.2cm@C=.5cm{
     & & \node & & \node & \\
     & & & \node\curledgeee[ul]\curledgeee[ur] & & \\
     & & & & & \node \\
     & \node\curledgeee[r] & \node\curledgeee[r]
       & \node\curledgeee[r]\curledgeee[uu] & \node\curledgeee[ur]\curledgeee[dr] & \\
     & & & & & \node \\
     & & & \node\curledgeee[uu] & & \\
     & & \node\curledgeee[ur]\ftxt<1.5ex,-1.5ex>{$u_{_{j}}$} & & & \\
     \node\curledgeee[r]
       & \node\curledgeee[ur]^{\lab{\zeta(i,j)}}\ftxt<1.7ex,-1.2ex>{$u_{_{i}}$}
       & & \node\curledgeee[uu] & & \\
     & & & & \node\curledgeee[ul] & \\
     & \node\curledgeee[uu] & & \node\curledgeee[uu] & & \\
  }
}

\def\diagfanar{
  \xymatrix@R=.5cm@C=.5cm{
 \node \curledge[rrrr] & & & & \node
\ftxt<-2.65cm,0.3cm>{\footnotesize$x$}
\ftxt<-1.35cm,0.35cm>{\footnotesize$\matr{L}_{_{k,n}}$}
\ftxt<0cm,0.3cm>{\footnotesize$y$}
  }
}
\def\diagT{
  \xymatrix@R=1cm@C=2cm{
    & \node
         \inedge<-.35cm,.3cm>\inedge<-.15cm,.3cm>\inedge<.35cm,.3cm>
         \ftxt<0cm,.5cm>{$\overbrace{\qquad}$}
         \ftxt<0cm,.8cm>{\footnotesize $R_{_{\{\widetilde{i},\widetilde{j}\}}}$} & \\  & \\
    \node\ftxt<-1.7ex,0ex>{$x$}
    \ftxt<14.8ex,14.3ex>{$u$}
    \ftxt<15.1ex,-14.5ex>{$v$}
     \curledge[uur]^{\lab{\matr{\matr{L}}_{_{i,n}}}}\curledge[ddr]_{\lab{\matr{L}_{_{j,n}}}}
      \ftxt<2.2cm,.3cm>{\footnotesize$\matr{L}_{_{k,n}}$}
       \ftxt<2.2cm,-.3cm>{\footnotesize$\forall{k} \in C_{_{\{i,j\}}}$}
      \curledge[rr]
            & & \node\ftxt<1.7ex,0ex>{$y$}\edge{0cm}[uul]\edge{0cm}[ddl]
      \\ \\
    & \node
         \inedge<-.35cm,-.3cm>\inedge<.15cm,-.3cm>\inedge<.35cm,-.3cm>
         \ftxt<0cm,-.5cm>{$\underbrace{\qquad}$}
         \ftxt<0cm,-.8cm>{\footnotesize $R_{_{\{\widetilde{i},\widetilde{j}\}}}$} & \\
  }
}
\def\diagCs{
  \xymatrix@R=.7cm@C=1cm{
    & \node
         \inedge<-.35cm,.3cm>\inedge<-.15cm,.3cm>\inedge<.35cm,.3cm>
         \ftxt<0cm,.5cm>{$\overbrace{\qquad}$}
         \ftxt<0cm,.8cm>{\footnotesize $R_{_{\{\widetilde{i},\widetilde{j}\}}}$} & \\
             \node\ftxt<-1.7ex,0ex>{$x$}
    \ftxt<14.8ex,14.3ex>{$u$}
    \ftxt<15.1ex,-14.5ex>{$v$}
     \curledge[ur]^{\lab{\matr{\matr{L}}_{_i}}}\curledge[dr]_{\lab{\matr{L}_{_j}}}
      \ftxt<2.2cm,.3cm>{\footnotesize$\matr{L}_{_{k}}$}
       \ftxt<2.2cm,-.3cm>{\footnotesize$\forall{k} \in S_{_{\{i,j\}}}$}
      \curledge[r]
            & & \node\ftxt<1.7ex,0ex>{$y$}\edge{0cm}[ul]\edge{0cm}[dl]
      \\
    & \node
         \inedge<-.35cm,-.3cm>\inedge<.15cm,-.3cm>\inedge<.35cm,-.3cm>
         \ftxt<0cm,-.5cm>{$\underbrace{\qquad}$}
         \ftxt<0cm,-.8cm>{\footnotesize $R_{_{\{\widetilde{i},\widetilde{j}\}}}$} & \\
  }
}
\def\diagchar{
  \xymatrix@R=.2cm@C=.5cm{
& & & &
\node\inedge<-.35cm,.3cm>\inedge<-.15cm,.3cm>\inedge<.35cm,.3cm>
                \ftxt<0cm,.4cm>{$\overbrace{\qquad}$}
                \ftxt<0.1cm,.65cm>{\footnotesize $R_{_{\{\widetilde{0},\widetilde{1}\}}}$}
                \ftxt<.11cm,-.2cm>{\footnotesize $w_{_{0}}$} \edge{0cm}[rrrrrl]
 & & & &
\node\inedge<-.35cm,.3cm>\inedge<-.15cm,.3cm>\inedge<.35cm,.3cm>
                \ftxt<0cm,.4cm>{$\overbrace{\qquad}$}
                \ftxt<0.1cm,.65cm>{\footnotesize $R_{_{\{\widetilde{0},\widetilde{1}\}}}$}
                \ftxt<0cm,-.2cm>{\footnotesize $u_{_{0}}$}
                \\ \\ \\
  &  & & &  \node\inedge<-.35cm,.3cm>\inedge<-.15cm,.3cm>\inedge<.35cm,.3cm>
                \ftxt<0cm,.4cm>{$\overbrace{\qquad}$}
                \ftxt<0.1cm,.65cm>{\footnotesize $R_{_{\{\widetilde{0},\widetilde{1}\}}}$} \curledge[rrrr]
                \ftxt<1.3cm,-.25cm>{ \footnotesize $\matr{L}_{_{0,n}}$}
                \ftxt<.1cm,-.2cm>{ \footnotesize $w_{_{1}}$}
 & & & &  \node\inedge<-.35cm,.3cm>\inedge<-.15cm,.3cm>\inedge<.35cm,.3cm>
                \ftxt<0cm,.4cm>{$\overbrace{\qquad}$}
               \ftxt<0.1cm,.65cm>{\footnotesize $R_{_{\{\widetilde{0},\widetilde{1}\}}}$}
                \ftxt<.1cm,-.2cm>{ \footnotesize $u_{_{1}}$}
                \\ \\ & \ftxt<3.2cm,-1.5cm>{$\vdots$} &\\
  &  & & &\node\inedge<-.35cm,.3cm>\inedge<-.15cm,.3cm>\inedge<.35cm,.3cm>
                \ftxt<0cm,.4cm>{$\overbrace{\qquad}$}
                \ftxt<.1cm,.65cm>{\footnotesize $R_{_{\{\widetilde{0},\widetilde{2}\}}}$} \curledge[rrrr]
                \ftxt<1.3cm,-.3cm>{ \footnotesize $\matr{L}_{_{0,n}}$}
                \ftxt<.1cm,-.2cm>{ \footnotesize $w_{_{2}}$}
 & & & &  \node\inedge<-.35cm,.3cm>\inedge<-.15cm,.3cm>\inedge<.35cm,.3cm>
               \ftxt<0cm,.4cm>{$\overbrace{\qquad}$}
               \ftxt<.1cm,.65cm>{\footnotesize $R_{_{\{\widetilde{0},\widetilde{1}\}}}$}
                   \ftxt<.1cm,-.2cm>{ \footnotesize $u_{_{2}}$}
                \\
    \node
    \curledge[uuuuuuurrrr]
    \ftxt<1.5cm,2cm>{ \footnotesize $\matr{L}_{_{0,n}}$}
     \ftxt<-.3cm,0cm>{ \footnotesize $x$}
     \ftxt<1.6cm,1.2cm>{ \footnotesize $\matr{L}_{_{1,n}}$}
     \ftxt<1.5cm,-.1cm>{ \footnotesize $\matr{L}_{_{2,n}}$}
     \ftxt<1.5cm,-1.7cm>{ \footnotesize $\matr{L}_{_{m-1,n}}$}
    \curledge[uuuurrrr]\curledge[urrrr]\curledge[ddddddrrrr]
                \\ \\ \\ \\ \\ \\
&  & &
&\node\inedge<-.35cm,-.3cm>\inedge<-.15cm,-.3cm>\inedge<.35cm,-.3cm>
                 \ftxt<0cm,-.4cm>{$\underbrace{\qquad}$}
                \ftxt<.1cm,-.7cm>{\footnotesize $R_{_{\{\widetilde{0},\widetilde{m-1}\}}}$} \curledge[rrrr]
                \ftxt<1.3cm,-.3cm>{ \footnotesize $\matr{L}_{_{0,n}}$}
                \ftxt<.15cm,.2cm>{ \footnotesize $w_{_{m-1}}$}
 & & & &  \node\inedge<-.35cm,-.3cm>\inedge<-.15cm,-.3cm>\inedge<.35cm,-.3cm>
               \ftxt<0cm,-.4cm>{$\underbrace{\qquad}$}
               \ftxt<.1cm,-.65cm>{\footnotesize $R_{_{\{\widetilde{0},\widetilde{1}\}}}$}
                   \ftxt<.1cm,.2cm>{ \footnotesize $u_{_{m-1}}$}
  }
}
\def\diagand{
  \xymatrix@R=.2cm@C=.5cm{
\node \curledge[dddrrr]
\inedge<-.35cm,.3cm>\inedge<-.15cm,.3cm>\inedge<.35cm,.3cm>
                \ftxt<0cm,.4cm>{$\overbrace{\qquad}$}
                \ftxt<-.1cm,.67cm>{\footnotesize $R_{_{\{\widetilde{0},\widetilde{1}\}}}$}
                \ftxt<0cm,-.2cm>{\footnotesize $x$}
                \\ \\ \\
 & & &  \node   \inedge<-.35cm,.3cm>\inedge<-.15cm,.3cm>\inedge<.35cm,.3cm>
                \ftxt<0cm,.4cm>{$\overbrace{\qquad}$}
                \ftxt<0cm,.67cm>{\footnotesize $R_{_{\{\widetilde{0},\widetilde{1},\widetilde{2}\}}}$}
                \curledge[rrrr]
                \ftxt<1.3cm,-.25cm>{ \footnotesize $\matr{L}_{_{1,n}}$}
                \ftxt<.1cm,-.2cm>{ \footnotesize $w$}
 & & & &  \node\inedge<-.35cm,.3cm>\inedge<-.15cm,.3cm>\inedge<.35cm,.3cm>
                \ftxt<0cm,.4cm>{$\overbrace{\qquad}$}
               \ftxt<0cm,.67cm>{\footnotesize $R_{_{\{\widetilde{0},\widetilde{1}\}}}$}
                \ftxt<0cm,-.2cm>{ \footnotesize $z$}
                \\ \\
\node
\inedge<-.35cm,-.3cm>\inedge<-.15cm,-.3cm>\inedge<.35cm,-.3cm>
 \ftxt<0cm,-.4cm>{$\underbrace{\qquad}$}
   \ftxt<-0.05cm,.23cm>{ \footnotesize $y$}
                \ftxt<0cm,-.67cm>{\footnotesize $R_{_{\{\widetilde{0},\widetilde{1}\}}}$}
 & & \node  \edge{0cm}[uur]
              \ftxt<.2cm,0cm>{ \footnotesize $v$}
                     \ftxt<-1.cm,.2cm>{ \footnotesize $-$}
                      \ftxt<-.92cm,.2cm>{ \footnotesize $-$}
                       \ftxt<-.72cm,.2cm>{ \footnotesize $-$}
                        \ftxt<-.52cm,.2cm>{ \footnotesize $-$}
                         \ftxt<-.47cm,.2cm>{ \footnotesize $-$}
   \ftxt<-1.0cm,-.2cm>{ \footnotesize $-$}
    \ftxt<-.92cm,-.2cm>{ \footnotesize $-$}
     \ftxt<-.72cm,-.2cm>{ \footnotesize $-$}
      \ftxt<-.52cm,-.2cm>{ \footnotesize $-$}
      \ftxt<-.47cm,-.2cm>{ \footnotesize $-$}
                    \ftxt<-.7cm,0.05cm>{ \footnotesize $\Cs_{m}^{+}$}
                   \ftxt<-1.09cm,0.042cm>{ \footnotesize $|$}
                   \ftxt<-1.09cm,-0.042cm>{ \footnotesize $|$}
                   \ftxt<-0.37cm,0.043cm>{ \footnotesize $|$}
                   \ftxt<-0.37cm,-0.043cm>{ \footnotesize $|$}
\ftxt<-1.4cm,0cm>{ \footnotesize $-$}
\ftxt<-1.2cm,0cm>{\footnotesize $-$}
\ftxt<-1.13cm,0cm>{\footnotesize$-$}
\ftxt<-.10cm,0cm>{\footnotesize$-$}
\ftxt<-.2cm,0cm>{\footnotesize$-$}
                        \ftxt<-.7cm,1.3cm>{ \footnotesize $\matr{L}_{_{0,n}}$}
     }
}
\def\diagD{
\begin{picture}(300.00,200.00)
\emline{50}{180}{1}{80.00}{180}{2}
\emline{80}{180}{1}{80.00}{100}{2}
\emline{80}{100}{1}{50.00}{100}{2}
\emline{50}{100}{1}{50.00}{180}{2} \put(40,140){\circle*{4}}
\put(55.00,135){$\Ch_{_m}$} \emline{40}{140}{1}{50}{140}{2}
\put(25.00,136){$x$} \put(90,170){\circle*{4}}
\emline{90}{170}{1}{80.00}{170}{2}

\put(90,160){\circle*{4.00}} \emline{90}{160}{1}{80.00}{160}{2}

\put(88.00,130){\vdots}

\put(90.00,110.00){\circle*{4.00}}
\emline{90.00}{110.00}{1}{80.00}{110.00}{2}
\emline{50.00}{90.00}{1}{80.00}{90.00}{2}
\emline{80.00}{90.00}{1}{80.00}{10.00}{2}
\emline{80.00}{10.00}{1}{50.00}{10.00}{2}
\emline{50.00}{10.00}{1}{50.00}{90.00}{2}
\put(40.00,50.00){\circle*{4.00}}
\put(55.00,45.00){$\Ex_{_m}$}\put(25.00,46){$y$}
\emline{40.00}{50.00}{1}{50.00}{50.00}{2}

\put(90.00,80.00){\circle*{4.00}}
\emline{90.00}{80.00}{1}{80.00}{80.00}{2}

\put(90.00,70.00){\circle*{4.00}}
\emline{90.00}{70.00}{1}{80.00}{70.00}{2}

\put(88.00,40){\vdots}

\put(90.00,20.00){\circle*{4.00}}
\emline{90.00}{20.00}{1}{80.00}{20.00}{2}
\emline{240.00}{140.00}{1}{270.00}{140.00}{2}
\emline{270.00}{140.00}{1}{270.00}{60.00}{2}
\emline{270.00}{60.00}{1}{240.00}{60.00}{2}
\emline{240.00}{60.00}{1}{240.00}{140.00}{2}
\put(280.00,100.00){\circle*{4.00}}
\put(245.00,95.00){$\Ps_{_m}$}\put(290.00,95){$z$}
\emline{270.00}{100.00}{1}{280.00}{100.00}{2}

\put(230.00,130.00){\circle*{4.00}}
\emline{230.00}{130.00}{1}{240.00}{130.00}{2}

\put(230.00,120.00){\circle*{4.00}}
\emline{230.00}{120.00}{1}{240.00}{120.00}{2}

\put(228.00,90){\vdots}

\put(230.00,70.00){\circle*{4.00}}
\emline{230.00}{70.00}{1}{240.00}{70.00}{2}
\emline{180.0}{160.00}{1}{215.00}{160.00}{2}
\emline{215.00}{160.00}{1}{215.00}{140.00}{2}
\emline{215.00}{140.00}{1}{180.00}{140.00}{2}
\emline{180.00}{140.00}{1}{180.0}{160.00}{2} \put(187,147){$\Oor_{_{m}}$}

\emline{180.0}{130.00}{1}{215.00}{130.00}{2}
\emline{215.00}{130.00}{1}{215.00}{110.00}{2}
\emline{215.00}{110.00}{1}{180.00}{110.00}{2}
\emline{180.00}{110.00}{1}{180.0}{130.00}{2} \put(182.00,117){$\Oand_{_{m}}$}

\emline{180.0}{50.00}{1}{215.00}{50.00}{2}
\emline{215.00}{50.00}{1}{215.00}{30.00}{2}
\emline{215.00}{30.00}{1}{180.00}{30.00}{2}
\emline{180.00}{30.00}{1}{180.0}{50.00}{2} \put(182.00,37){$\Oand_{_{m}}$}
\put(195.00,75){\vdots}
\emline{90.0}{80.00}{1}{135.00}{147.00}{2}
\put(173.00,147.00){\circle*{4.00}}

\emline{135.00}{147.00}{1}{140.00}{147.00}{2}
\put(143.00,151){$_{_{\Onot_{_{m}}}}$}
\emline{165.00}{147.00}{1}{180.00}{147.00}{2}

\emline{141.0}{155.00}{1}{165.00}{155.00}{2}
\emline{165.00}{155.00}{1}{165.00}{143.00}{2}
\emline{165.00}{143.00}{1}{141.00}{143.00}{2}
\emline{141.00}{143.00}{1}{141.0}{155.00}{2}

\emline{90.0}{70.00}{1}{180.00}{117.00}{2}
\emline{90.0}{20.00}{1}{180.00}{35.00}{2}

\emline{90.0}{170.00}{1}{180.00}{156.00}{2}
\emline{90.0}{160.00}{1}{180.00}{126.00}{2}
\emline{90.0}{110.00}{1}{180.00}{44.00}{2}

\emline{215.00}{150.00}{1}{230.0}{130.00}{2}
\emline{215.00}{120.00}{1}{230.0}{120.00}{2}
\emline{215.00}{40.00}{1}{230.0}{70.00}{2}

\end{picture}
}
\def\diagP{
\begin{picture}(300.00,250.00)
\emline{50}{280}{1}{80.00}{280}{2}
\emline{80}{280}{1}{80.00}{200}{2}
\emline{80}{200}{1}{50.00}{200}{2}
\emline{50}{200}{1}{50.00}{280}{2} \put(40,240){\circle*{4}}
\put(55.00,235){$\Ch_{_m}$} \emline{40}{240}{1}{50}{240}{2}
\put(22.00,236){$x_{_0}$}
\put(90,270){\circle*{4}} \emline{90}{270}{1}{80.00}{270}{2}

\put(90,260){\circle*{4.00}} \emline{90}{260}{1}{80.00}{260}{2}

\put(88.00,230){\vdots}

\put(90.00,210.00){\circle*{4.00}}
\emline{90.00}{210.00}{1}{80.00}{210.00}{2}
\emline{50}{190}{1}{80.00}{190}{2}
\emline{80}{190}{1}{80.00}{110}{2}
\emline{80}{110}{1}{50.00}{110}{2}
\emline{50}{110}{1}{50.00}{190}{2} \put(40,150){\circle*{4}}
\put(55.00,145){$\Ch_{_m}$} \emline{40}{150}{1}{50}{150}{2}
\put(22.00,146){$x_{_1}$} \put(90,180){\circle*{4}}
\emline{90}{180}{1}{80.00}{180}{2}

\put(90,170){\circle*{4.00}} \emline{90}{170}{1}{80.00}{170}{2}

\put(88.00,140){\vdots}

\put(90.00,120.00){\circle*{4.00}}
\emline{90.00}{120.00}{1}{80.00}{120.00}{2}
\emline{50.00}{60.00}{1}{80.00}{60.00}{2}
\emline{80.00}{60.00}{1}{80.00}{-20.00}{2}
\emline{80.00}{-20.00}{1}{50.00}{-20.00}{2}
\emline{50.00}{-20.00}{1}{50.00}{60.00}{2}
\put(40.00,20.00){\circle*{4.00}}
\put(55.00,15.00){$\Ch_{_m}$}\put(8.00,16){$x_{_{m-1}}$}

\emline{40.00}{20.00}{1}{50.00}{20.00}{2}

\put(90.00,50.00){\circle*{4.00}}
\emline{90.00}{50.00}{1}{80.00}{50.00}{2}

\put(90.00,40.00){\circle*{4.00}}
\emline{90.00}{40.00}{1}{80.00}{40.00}{2}

\put(88.00,10){\vdots}

\put(90.00,-10.00){\circle*{4.00}}
\emline{90.00}{-10}{1}{80.00}{-10.00}{2}
\emline{220.00}{170.00}{1}{250.00}{170.00}{2}
\emline{250.00}{170.00}{1}{250.00}{90.00}{2}
\emline{250.00}{90.00}{1}{220.00}{90.00}{2}
\emline{220.00}{90.00}{1}{220.00}{170.00}{2}
\put(260.00,130.00){\circle*{4.00}}
\put(225.00,125.00){$\Ps_{_m}$}\put(270,125){t}
\emline{250.00}{130.00}{1}{260.00}{130.00}{2}

\put(210.00,160.00){\circle*{4.00}}
\emline{210.00}{160.00}{1}{220.00}{160.00}{2}

\put(210.00,150.00){\circle*{4.00}}
\emline{210.00}{150.00}{1}{220.00}{150.00}{2}

\put(208.00,120){\vdots}

\put(210.00,100.00){\circle*{4.00}}
\emline{210.00}{100.00}{1}{220.00}{100.00}{2}
\emline{150.0}{230.00}{1}{195.00}{230.00}{2}
\emline{195.00}{230.00}{1}{195.00}{180.00}{2}
\emline{195.00}{180.00}{1}{150.00}{180.00}{2}
\emline{150.00}{180.00}{1}{150.0}{230.00}{2}
\put(154.20,197){$\Vand_{_m}$}
\emline{150.0}{170.00}{1}{190.00}{170.00}{2}
\emline{190.00}{170.00}{1}{190.00}{120.00}{2}
\emline{190.00}{120.00}{1}{150.00}{120.00}{2}
\emline{150.00}{120.00}{1}{150.0}{170.00}{2}
\put(158.00,140){$\Vor_{_m}$}

\put(168.00,97){\vdots}

\emline{150.0}{80.00}{1}{190.00}{80.00}{2}
\emline{190.00}{80.00}{1}{190.00}{30.00}{2}
\emline{190.00}{30.00}{1}{150.00}{30.00}{2}
\emline{150.00}{30.00}{1}{150.0}{80.00}{2}
\put(158.00,50){$\Vor_{_m}$}
\emline{90.0}{270.00}{1}{150.00}{223.00}{2}
\emline{90.0}{260.00}{1}{150.00}{159.00}{2}
\emline{90.0}{210.00}{1}{150.00}{67.00}{2}
\put(63,80){\vdots}\put(143,131){\vdots}\put(143,196){\vdots}\put(143,45){\vdots}

\emline{90.0}{180.00}{1}{150.00}{221.00}{2}
\emline{90.0}{170.00}{1}{150.00}{157.00}{2}
\emline{90.0}{120.00}{1}{150.00}{65.00}{2}

\emline{90.0}{50.00}{1}{150.00}{187.00}{2}
\emline{90.0}{40.00}{1}{150.00}{127.00}{2}
\emline{90.0}{-10.00}{1}{150.00}{35.00}{2}

\emline{195.0}{205.00}{1}{210.00}{160.00}{2}
\emline{190.0}{145.00}{1}{210.00}{150.00}{2}
\emline{190.0}{55.00}{1}{210.00}{100.00}{2}

\end{picture}
}
\def\diagMux{
\begin{picture}(300.00,200.00)

\put(-52.00,85){$\begin{array}{ll} &\quad \left \{
\begin{array}{ll}
                                                              \\
                                                              \\
                                                              \\
                                                              \\
                                                              \\
                                                              \\
                                                              \\
                                                              \\
                                                              \\
                                                              \end{array}\right.\\
&\end{array} $}

\put(10.00,200.00){$\shortmid$} \put(300.00,200.00){$\shortmid$}
\put(10.00,190.00){$\shortmid$} \put(300.00,190.00){$\shortmid$}
\put(10.00,180.00){$\shortmid$} \put(300.00,180.00){$\shortmid$}
\put(10.00,170.00){$\shortmid$} \put(300.00,170.00){$\shortmid$}
\put(10.00,160.00){$\shortmid$} \put(300.00,160.00){$\shortmid$}
\put(10.00,150.00){$\shortmid$} \put(300.00,150.00){$\shortmid$}
\put(10.00,140.00){$\shortmid$} \put(300.00,140.00){$\shortmid$}
\put(10.00,130.00){$\shortmid$} \put(300.00,130.00){$\shortmid$}
\put(10.00,120.00){$\shortmid$} \put(300.00,120.00){$\shortmid$}
\put(10.00,110.00){$\shortmid$} \put(300.00,110.00){$\shortmid$}
\put(10.00,100.00){$\shortmid$} \put(300.00,100.00){$\shortmid$}
\put(10.00,90.00){$\shortmid$} \put(300.00,90.00){$\shortmid$}
\put(10.00,80.00){$\shortmid$} \put(300.00,80.00){$\shortmid$}
\put(10.00,70.00){$\shortmid$} \put(300.00,70.00){$\shortmid$}
\put(10.00,60.00){$\shortmid$} \put(300.00,60.00){$\shortmid$}
\put(10.00,50.00){$\shortmid$} \put(300.00,50.00){$\shortmid$}
\put(10.00,40.00){$\shortmid$} \put(300.00,40.00){$\shortmid$}
\put(10.00,30.00){$\shortmid$} \put(300.00,30.00){$\shortmid$}
\put(10.00,20.00){$\shortmid$} \put(300.00,20.00){$\shortmid$}
\put(10.00,10.00){$\shortmid$} \put(300.00,10.00){$\shortmid$}
\put(10.00,00.00){$\shortmid$} \put(300.00,00.00){$\shortmid$}
\put(10.00,200.00){- - - -  - - - - - - - - - - - - - - - - - - - -
- - - - -  } \put(10.00,00.00){- - - -  - - - - - - - - - - - - - -
- - - - - - - - - - -   }
\emline{150.0}{195.00}{1}{190.00}{195.00}{2}
\emline{190.00}{195.00}{1}{190.00}{165.00}{2}
\emline{190.00}{165.00}{1}{150.00}{165.00}{2}
\emline{150.00}{165.00}{1}{150.0}{195.00}{2} \put(157,175){$\Dot_{_{m}}$}

\emline{0.0}{170.00}{1}{150.00}{170.00}{2}
\emline{125.0}{190.00}{1}{150.00}{190.00}{2}
\put(0,170){\circle*{4.00}}
 \emline{190}{180.00}{1}{210.00}{180}{2}
\put(210,180.00){\circle*{4.00}}

\emline{150.0}{155.00}{1}{190.00}{155.00}{2}
\emline{190.00}{155.00}{1}{190.00}{125.00}{2}
\emline{190.00}{125.00}{1}{150.00}{125.00}{2}
\emline{150.00}{125.00}{1}{150.0}{155.00}{2}
\put(157,135){$\Dot_{_{m}}$}

\emline{0.0}{130.00}{1}{150.00}{130.00}{2}
\emline{102}{150.00}{1}{150.00}{150.00}{2}
\put(0.00,130.00){\circle*{4.00}}

\emline{190}{140.00}{1}{210.00}{140.00}{2}
\put(210,140.00){\circle*{4.00}}
\emline{150.0}{50.00}{1}{190.00}{50.00}{2}
\emline{190.00}{50.00}{1}{190.00}{20.00}{2}
\emline{190.00}{20.00}{1}{150.00}{20.00}{2}
\emline{150.00}{20.00}{1}{150.0}{50.00}{2} \put(157,30){$\Dot_{_{m}}$}

\emline{0.0}{25.00}{1}{150.00}{25.00}{2}
\emline{50.0}{45.00}{1}{150.00}{45.00}{2}
\put(0.00,25.00){\circle*{4.00}}

\emline{190}{30.00}{1}{210.00}{30.00}{2}
\put(210,30.00){\circle*{4.00}}

\emline{250.0}{140.00}{1}{290.00}{140.00}{2}
\emline{290.00}{140.00}{1}{290.00}{90.00}{2}
\emline{290.00}{90.00}{1}{250.00}{90.00}{2}
\emline{250.00}{90.00}{1}{250.0}{140.00}{2}
\put(263.00,110){$\Ext_{_m}$} \put(315.00,110){$t$}

\emline{210.0}{180.00}{1}{250.00}{135.00}{2}
\emline{210.0}{140.00}{1}{250.00}{128.00}{2}
\emline{210.0}{30.00}{1}{250.00}{100.00}{2}

\emline{290.0}{115.00}{1}{310.00}{115.00}{2}\put(310.00,115.00){\circle*{4.00}}
\emline{125}{220}{1}{125.00}{190.00}{2}\put(125.00,220.00){\circle*{4.00}}
\emline{102.0}{220.00}{1}{102.00}{150.00}{2}\put(102.00,220.00){\circle*{4.00}}
\emline{50.0}{220.00}{1}{50.00}{45.00}{2}\put(50.00,220.00){\circle*{4.00}}

\put(-28,95){$\vect{v}$}

\put(41,227){$\overbrace{         \ \ \ \ \ \ \ \ \ \ \ \ \ \ \ \ \
\ }$}

 \put(83,240){$\vect{u}$}

\put(2,70){\vdots}\put(170,85){\vdots}  \put(245,105){\vdots}

\put(70,218){\ldots}

\end{picture}
}
\def\diagDel{
\begin{picture}(300.00,200.00)

\emline{200.0}{150.00}{1}{280.00}{150.00}{2}
\emline{280.00}{150.00}{1}{280.00}{125.00}{2}
\emline{280.00}{125.00}{1}{200.00}{125.00}{2}
\emline{200.00}{125.00}{1}{200.0}{150.00}{2}
\put(228,133){$\Ch_{_m}$}\put(225,112){\ldots}

\emline{0.0}{165.00}{1}{240.00}{165.00}{2}
\emline{240}{165.00}{1}{240.00}{150.00}{2}
\put(0.00,165.00){\circle*{4.00}}

\emline{270}{125.00}{1}{270.00}{100.00}{2}
\emline{255}{125.00}{1}{255.00}{100.00}{2}
\emline{210}{125.00}{1}{210.00}{100.00}{2}

\put(270,113.00){\circle*{4.00}} \put(255,113.00){\circle*{4.00}}
\put(210,113.00){\circle*{4.00}}
\emline{30.0}{35.00}{1}{55.00}{35.00}{2}
\emline{55.00}{35.00}{1}{55.00}{16.00}{2}
\emline{55.00}{16.00}{1}{30.00}{16.00}{2}
\emline{30.00}{16.00}{1}{30.0}{35.00}{2} \put(38.00,20){$F$}
\emline{150.0}{35.00}{1}{175.00}{35.00}{2}
\emline{175.00}{35.00}{1}{175.00}{16.00}{2}
\emline{175.00}{16.00}{1}{150.00}{16.00}{2}
\emline{150.00}{16.00}{1}{150.0}{35.00}{2} \put(156.00,20){$F$}

\emline{200.0}{100.00}{1}{280.00}{100.00}{2}
\emline{280.00}{100.00}{1}{280.00}{20.00}{2}
\emline{280.00}{20.00}{1}{200.00}{20.00}{2}
\emline{200.00}{20.00}{1}{200.0}{100.00}{2}
\put(225.00,60){$\Mux_{_m}$} \put(193,52){$\vdots$}
\put(93,27){\ldots}

\emline{280.00}{65.00}{1}{315.0}{65.00}{2}
\put(315,65.00){\circle*{4.00}}

\emline{20.0}{93.00}{1}{200.00}{93.00}{2}
\emline{90.00}{77.00}{1}{200.00}{77.00}{2}
\emline{175.00}{27.00}{1}{200.00}{27.00}{2}

\emline{20.0}{93.00}{1}{0.00}{27.00}{2}
\emline{90.00}{77.00}{1}{75.00}{27.00}{2}

\emline{132.00}{27.00}{1}{135.00}{42.00}{2}
\emline{135.00}{42.00}{1}{200.00}{42.00}{2}

\put(75,27){\circle*{4.00}}
\put(132,27){\circle*{4.00}}\put(188,27){\circle*{4.00}}
\emline{00.0}{27.00}{1}{20.00}{27.00}{2} \put(0,27){\circle*{4.00}}
\emline{20.0}{27.00}{1}{30.00}{27.00}{2}

\emline{55.0}{27.00}{1}{85.00}{27.00}{2}

\emline{120.0}{27.00}{1}{150.00}{27.00}{2}

\put(320.00,60){$t$} \put(-15.00,161){$y$}  \put(-15.00,15){$v_{_{0}}$}
                                            \put(70.00,15){$v_{_{1}}$}
                                            \put(118.00,15){$v_{_{m-2}}$}
                                            \put(180.00,15){$v_{_{m-1}}$}
\end{picture}
}

\newtheorem{precor}{{\bf Corollary}}

\newtheorem{precon}{{\bf Conjecture}}

\newtheorem{predefin}{{\bf Definition}}

\newenvironment{defin}[1]{\begin{predefin}{\hspace{-0.5
                   em}{\bf.\ }}{\rm
#1}\hfill{$\spadesuit$}}{\end{predefin}}
\newtheorem{preexm}{{\bf Example}}

\newtheorem{preappl}{{\bf Application}}

\newtheorem{prelem}{{\bf Lemma}}

\newenvironment{lem}{\begin{prelem}{\hspace{-0.5
               em}{\bf.\ }}}{\end{prelem}}
\newtheorem{preproof}{{\bf Proof.\ }}

\newenvironment{proof}[1]{\begin{preproof}{\rm
               #1}\hfill{$\blacksquare$}}{\end{preproof}}
\newtheorem{presproof}{{\bf Sketch of Proof.\ }}

\newtheorem{prethm}{{\bf Theorem}}

\newenvironment{thm}{\begin{prethm}{\hspace{-0.5
               em}{\bf.\ }}}{\end{prethm}}
\newtheorem{prealphthm}{{\bf Theorem}}

\newtheorem{prepro}{{\bf Proposition}}

\newenvironment{pro}{\begin{prepro}{\hspace{-0.5
               em}{\bf.\ }}}{\end{prepro}}
\newtheorem{preprb}{{\bf Problem}}

\def\eg {e}
\def\Kv {\widetilde}

\def\Lkn {{ \matr{L}_{_{k,n}}}[x,y;R]}
\def\Id {{\sf id}}
\def\Ch {{\sf ch}}
\def\Ps {{\sf ps}}
\def\Cs {{\sf cs}}

\def\Ex {{\sf xp}}
\def\Ext {{\sf xt}}
\def\Onot {{\sf not}}
\def\Oand {{\sf and}}
\def\Oor {{\sf or}}
\def\Vand {{\sf vand}}
\def\Vor {{\sf vor}}
\def \Mux {{\sf Mux}}
\def \Edg {{\sf Edg}}
\def \Dot {{\sf dot}}
\def\Add {{\sf add}}
\def\Sub {{\sf sub}}

\def\isdef{\mbox {$\ \stackrel{\rm def}{=} \ $}}

\begin{document}
\footnotetext[1]{Correspondence should be addressed to {\tt
daneshgar@sharif.ir}.}
\begin{center}
{\Large \bf  Graph Coloring and Function Simulation}\\
\vspace*{0.5cm}
{\bf Amir Daneshgar}\\
{\it Department of Mathematical Sciences} \\
{\it Sharif University of Technology} \\
{\it P.O. Box {\rm 11155--9415}, Tehran, Iran,}\\ \ \\
{\bf Ali Reza Rahimi}\\
{\it Institute of Mathematics Research} \\
{\it Tarbiat Moalem University} \\
{\it Tehran, Iran,}\\ \ \\
{\bf Siamak Taati}\\
{\it Universit\'e de Nice-Sophia Antipolis\\
     Laboratoire I$3$S, $2000$, route des Lucioles\\
     $06903$ Sophia Antipolis, France.
} \\ \ \\
\end{center}
\begin{abstract}

We prove that every partial function with finite domain and range can be effectively
simulated through sequential colorings of graphs.
Namely, we show that given a finite set $S=\{0,1,\ldots,m-1\}$ and a number $n \geq \max\{m,3\}$, any partial function
$\varphi:S^{^p} \to S^{^q}$ (i.e. it may not be defined on some elements of its domain $S^{^p}$) can be effectively (i.e. in polynomial time) transformed to a simple graph $\matr{G}_{_{\varphi,n}}$
 along with three sets of specified vertices
$$X = \{x_{_{0}},x_{_{1}},\ldots,x_{_{p-1}}\}, \ \ Y = \{y_{_{0}},y_{_{1}},\ldots,y_{_{q-1}}\}, \ \ R = \{\Kv{0},\Kv{1},\ldots,\Kv{n-1}\},$$
such that any assignment $\sigma_{_{0}}: X  \cup R \to \{0,1,\ldots,n-1\} $
with
$\sigma_{_{0}}(\Kv{i})=i$ for all $0 \leq i < n$, is {\it
uniquely} and {\it effectively} extendable to a proper $n$-coloring $\sigma$ of
$\matr{G}_{_{\varphi,n}}$ for which we have
$$\varphi(\sigma(x_{_{0}}),\sigma(x_{_{1}}),\ldots,\sigma(x_{_{p-1}}))=(\sigma(y_{_{0}}),\sigma(y_{_{1}}),\ldots,\sigma(y_{_{q-1}})),$$
unless $(\sigma(x_{_{0}}),\sigma(x_{_{1}}),\ldots,\sigma(x_{_{p-1}}))$ is not in the domain of $\varphi$ (in which case $\sigma_{_{0}}$  has no extension to a proper $n$-coloring of $\matr{G}_{_{\varphi,n}}$).
\vspace*{.7 cm}\\
{\bf Index Words:}\ {\footnotesize sequential graph coloring, uniquely colorable graphs, defining sets, graph grammars/amalgams, function evaluation, quantum computation.}\\
{\bf MR Subject Classification:}\ 05C15.\\
\end{abstract}
\setcounter{footnote}{1}
\section{Introduction}\label{INTRO}
Graph coloring, as a special case of the graph homomorphism
problem, is widely known to be among the most fundamental problems
in graph theory and combinatorics.
Based on some recent deep
non-approximability results
(see e.g. \cite{ARO09,GON07,GHS02,HENE04,JETO95,RASU07} and references therein),
it is known that the coloring problem is among the hardest problems in
${\bf NP}$ and does not admit an efficient solution/approximation,
unless ${\bf P}={\bf NP}$ (or a similar coincidence
for randomized classes as ${\bf RP}$). \\
Although the hardness of the graph coloring problem causes troubles
in applications that require finding graph colorings,
it is desirable in other applications,
such as cryptography and data security, where
such hardness conditions are sought.
In cryptography, hardness properties are particularly appreciated
when trapdoor keys are also available.\\
The idea of using combinatorial structures in cryptography has already been studied
by a number of authors (e.g. see \cite{BLA97,CGS98,CDS94,EUG92,FNX04,FILA06,KUKO05})
who, among other things, have used such structures to construct
secret sharing schemes.
Although, the idea of using combinatorial structures and their hardness results seems to be fruitful in cryptography, most of the cryptographic schemes introduced so far do not satisfy real-world security or efficiency conditions
that are needed in real applications.\\
The idea of using graph colorings and their hardness results in cryptography
has also been studied by many contributors in the field.
Using uniquely colorable graphs in this context seems to go back
to the ideas of the first author of this article who already hinted
at this in~\cite{DAN01}.
As the implementation of such applications needs
a well-understood theory of computation in terms of graph colorings,
it sounds reasonable to develop such a theory first, before putting such an application in the machinery (see \cite{DE08}).\\
This article can be considered as a sequel to \cite{DE08} and \cite{DHT?}, in which we try to complete the scenario of computability in terms of
graph colorings as much as it is needed in the above-mentioned context (which is restricted to the case of functions with finite domain and range).
It is worth noting that due to the hardness properties of the graph coloring problem, it is by no means straightforward to see how one may construct
efficiently colorable graphs
that can compute functions. The necessary background has already been developed in~\cite{DE08,DHT?}.  In this article, as our main result, we introduce a general setup in which one may efficiently simulate functions (with finite domain and range) using graphs and their colorings. The cryptographic applications of these results will appear elsewhere.\\
Roughly speaking,
our main result (Theorems~\ref{thrm:main} and~\ref{thrm:effectivealg}) states that any given partial function $\varphi:S^{^p} \to S^{^q}$
 (where $S \isdef \{0,1,\ldots,m-1\}$ is a finite set of symbols and $p$ and $q$ are positive integers)
can be effectively simulated by a simple graph $\matr{G}_{_{\varphi,n}}$ and its $n$-colorings ($n \geq \max\{m,3\}$) in the
following sense. The graph $\matr{G}_{_{\varphi,n}}$ has three subsets of designated vertices: the input
vertices $X = \{x_{_{0}},x_{_{1}},\ldots,x_{_{p-1}}\}$, the output vertices $Y = \{y_{_{0}},y_{_{1}},\ldots,y_{_{q-1}}\}$, and the
reference vertices $R = \{\Kv{0},\Kv{1},\ldots,\Kv{n-1}\}$.
The reference vertices induce a complete graph in $\matr{G}_{_{\varphi,n}}$, and their role is to establish a
correspondence between the colors and the elements of $S$. In every possible $n$-coloring
of $\matr{G}_{_{\varphi,n}}$, the reference vertices take distinct colors. For each $0 \leq i < m$, the color of
the vertex $\Kv{i} \in R$ stands for the element $i \in S$. Since the naming of the colors is
irrelevant, we often choose the color set $C \isdef \{0,1,\ldots,n-1\}$ and restrict ourselves
to the colorings in which every reference vertex $\Kv{i}$ takes the corresponding color $i$.
Now consider the proper $n$-colorings of the graph $\matr{G}_{_{\varphi,n}}$. The graph $\matr{G}_{_{\varphi,n}}$ has the property
that, in reference to the colors of the vertices in $R$, the constraint induced by $\matr{G}_{_{\varphi,n}}$ on
the colors of the vertices in $X$ and $Y$ corresponds precisely with the calculation of $\varphi$.
Namely, if for each $0 \leq i < n$ we color the vertex $\Kv{i}$ with $\sigma_{_{0}}(\Kv{i})=i$, and for each
$0 \leq j < p$ we choose an arbitrary color $\sigma_{_{0}}(x_{_{j}})$ for $x_{_{j}}$, then this assignment
of colors can be extended to a proper coloring $\sigma$ of the whole $\matr{G}_{_{\varphi,n}}$ using an appropriate universal sequential algorithm, if and only if
$(\sigma_{_{0}}(x_{_{0}}),\sigma_{_{0}}(x_{_{1}}),\ldots,\sigma_{_{0}}(x_{_{p-1}}))$ is in the domain of $\varphi$, in which case the extension is
unique and we have
$$\varphi(\sigma(x_{_{0}}),\sigma(x_{_{1}}),\ldots,\sigma(x_{_{p-1}}))=(\sigma(y_{_{0}}),\sigma(y_{_{1}}),\ldots,\sigma(y_{_{q-1}})).$$
Let us emphasize few aspects of this construction:
\begin{itemize}
\item{The construction of the graph $\matr{G}_{_{\varphi,n}}$ from the function $\varphi$ is effective (i.e., can be obtained using a
polynomial-time algorithm (see Section~\ref{sec:main})).}
\item{The number of colors $n$ can be chosen as small as $\max\{m,3\}$ (see Theorem~\ref{thrm:main}).}
\item{The coloring extensions are unique, and can be found using an efficient (i.e. polynomial-time) sequential coloring algorithm (see Theorem~\ref{thrm:effectivealg}).}
\item{With the above notion of simulation, the graphs can be combined readily to
implement function compositions. For example, if a function $\varphi$ is simulated by
a graph $\matr{G}_{_{\varphi,n}}$ and a function $\psi$  is simulated by a graph $\matr{G}_{_{\psi,n}}$, then the composition
  $\psi \circ \varphi$ can be implemented by simply putting $\matr{G}_{_{\psi,n}}$ and $\matr{G}_{_{\varphi,n}}$ next to each other,
identifying their reference vertices, and identifying the output vertices of $\matr{G}_{_{\varphi,n}}$
with the input vertices of $\matr{G}_{_{\psi,n}}$  (see Lemma~\ref{lem:combination} below).}
\end{itemize}
In \cite{DHT?} it is proved that Boolean functions (i.e. the case $m=2$ and $q=1$) can be effectively  simulated through $3$-colorings.
In this paper besides generalizing this result for any given function (as mentioned above), we have chosen an approach (among many other possible methods) in which we
explicitly show how to simulate basic Boolean and arithmetic operations
in the above sense (to be made precise in the sequel).
Our approach can be described as follows.  We first prove our main result for permutations through a constructive approach and next
we apply an extension result (see Lemma~\ref{lem:geninvfunc}) similar to what is usually done in quantum computing (e.g. see \cite{NICH00}), to generalize this result to arbitrary functions.  To simulate multivariate functions, we show how to simulate the coloring space of a graph using a large number of colors with the coloring space of a more elaborate graph using a limited number of colors.\\
In the rest of this section we go through some necessary prerequisites. In Section~\ref{sec:invertiblefuncs} we prove our result for permutations.
Next, in Section~\ref{sec:gadgets} we show that most important basic Boolean and arithmetic operations can be effectively simulated through colorings of graphs, and in Section~\ref{sec:main} we prove our main theorems.
\subsection{Simulation of functions by graphs: a starter}
 In this paper, we consider finite simple graphs and their
vertex colorings. Given a graph $\matr{G}=(V(\matr{G}),E(\matr{G}))$
on the vertex set $V(\matr{G})$ and with the edge set $E(\matr{G})$,
a proper $C$-coloring of $\matr{G}$ is a mapping
$\sigma: V(\matr{G}) \to C$ such that  for every edge $uv \in
E(\matr{G})$ we have $\sigma(u) \not = \sigma(v)$.
We often do not care about the nature of the colors, and when $|C|=n$, we
may simply talk about proper $n$-colorings
(or $n$-colorings, for short) rather than $C$-colorings. The smallest integer $n$ for
which the graph $\matr{G}$ admits an $n$-coloring is called the
chromatic number of $\matr{G}$
and is denoted by $\chi(\matr{G})$ (e.g. see \cite{JETO95,WES96} for the basic concepts of graph theory and graph colorings).\\
In what follows vectors and ordered lists are denoted by bold small symbols as $\vect{u}=(u_{_{0}},\ldots,u_{_{n-1}})$. We also fix the following notations:
$$\vect{0}_{_{n}} \isdef (\underbrace{0,0,\ldots,0}_{n~{\rm times}}), \quad  \vect{1}_{_{n}} \isdef (\underbrace{1,1,\ldots,1}_{n~{\rm times}}) \quad
{\rm and} \quad \vect{1}_{_{n}}^{^j} \isdef (\underbrace{\overbrace{0,0,\ldots,0}^{j~{\rm times}},1,0,\ldots,0}_{n~{\rm components}}).$$
Subsets in subscripts are used to show exclusion. In this regard, $S_{_{A}}\isdef S-A$ for any $A \subseteq S$ and
$$
\vect{u}_{_{\{k\}}} \isdef \left \{
   \begin{array}{ll}
   (u_{_{1}},u_{_{2}},\ldots,u_{_{{n-1}}}) &
   		\text{if $k=0$,} \\
   (u_{_{0}},\ldots,u_{_{k-1}},u_{_{k+1}},\ldots,u_{_{{n-1}}}) &
   		\text{if $0<k<n-1$,} \\
   (u_{_{0}},u_{_{1}},\ldots,u_{_{{n-2}}}) &
   		\text{if $k=n-1$.}
   \end{array}\right.
$$
Hereafter, we shall consider the set of all one-to-one and onto (i.e.
invertible) functions  from $S$ to $S$ as the symmetric group of all
$n$-permutations on $n$ elements.  In this setting
$\tau(i,j)\isdef (0)(1)(2)\cdots(ij)\cdots(n-1)$ denotes a transposition
that maps $i$ to $j$ and vice versa and keeps all the other elements unchanged.
Also, $\Id_{_{m}}$ stands for the identity function on $S$.\\
Throughout this article, we always assume
that $p, q, n$ and $m$ are natural number such that $m \geq 2$ and $n \geq \max\{m,3\}$.
We use
$$C \isdef \{0,1,\ldots,n-1\}, \quad S \isdef \{0,1,\ldots,m-1\}$$
as the sets of colors and symbols, respectively.
The sets
$$X \isdef \{x_{_{0}},x_{_{1}},\ldots,x_{_{p-1}}\}, \ \ Y \isdef \{y_{_{0}},y_{_{1}},\ldots,y_{_{q-1}}\}, \ \ R \isdef \{\Kv{0},\Kv{1},\ldots,\Kv{n-1}\}$$
are also interpreted as the sets of input, output, and reference
vertices, respectively (cf. below).
\begin{defin}{\label{defin:functionsim}
Let $p\geq 1$, $q\geq 1$ and also let
$\varphi:S^{^p} \to S^{^q}$ be a given
{\it partial} function (i.e. $\varphi$ may not be defined
on some of the elements of its domain as it is usually defined in theory of computation; see e.g.~\cite{BRI94}).
 We say that a graph $\matr{G}_{_{\varphi,n}}$ {\it simulates}
 the function $\varphi$ through $n$-colorings, if
$\matr{G}_{_{\varphi,n}}$ is a graph on the vertex set
$V(\matr{G}_{_{\varphi,n}})$ such that
\begin{itemize}
\item[i)]{The graph $\matr{G}_{_{\varphi,n}}$ is $n$-colorable
(i.e. $\chi(\matr{G}_{_{\varphi,n}}) \leq n$), and
$$X \cup Y \cup R \subseteq V(\matr{G}_{_{\varphi,n}})\;.$$}
\item[ii)]{In any $n$-coloring of $\matr{G}_{_{\varphi,n}}$ as $\sigma$,
vertices in $R$ are forced to take
different colors. (Therefore, without loss of generality we may
assume that for all $0 \leq i < n$ we have
$\sigma(\Kv{i})=i$.)}
\item[iii)]{Any assignment of  colors
$\sigma_{_{0}}: X \cup R \to C$,
with
$\sigma_{_{0}}(\Kv{i})=i$ for all $0 \leq i < n$,
has an extension to a proper $C$-coloring of the whole graph
$\matr{G}_{_{\varphi,n}}$ as $\sigma$
if and only if
$(\sigma(x_{_{0}}),\sigma(x_{_{1}}),\ldots,\sigma(x_{_{p-1}}))$ is in the domain of $\varphi$,
in which case the extension $\sigma$ is unique and
$$\varphi(\sigma(x_{_{0}}),\sigma(x_{_{1}}),\ldots,\sigma(x_{_{p-1}}))=(\sigma(y_{_{0}}),\sigma(y_{_{1}}),\ldots,\sigma(y_{_{q-1}})).$$
}
\end{itemize}
}\end{defin}
The following lemma easily follows from the definitions.  We explicitly state it for further reference.
\begin{lem}\label{lem:combination}
Let $\varphi_{_{i}}: S^{^p} \to S^{^{q_{_{i}}}}\;\; (i=1,2,\ldots,k)$
and $\psi: S^{^q} \to S^{^r}$ be partial functions, where
$p_{_i}$, $q_{_i}$, $q$ and $r$ are positive integers
with $\sum_{i=1}^k q_i = q$.
If for some $n\geq 3$, the functions $\varphi_{_i}$ and $\psi$
can be simulated through $n$-colorings,
the composition
$$\theta \isdef \psi \circ (\varphi_{_1},\varphi_{_2},\ldots,\varphi_{_k})$$
can also be simulated through $n$-colorings.
\end{lem}
It is worth noting that if a graph $\matr{G}[x,y]$ simulates a one-to-one 
partial function $\psi$, then the same graph with the input and the output swapped, $\matr{G}[y,x]$, simulates the inverse function $\psi^{-1}$.
This is true even if $\psi$ is not onto, in which case $\psi^{-1}$
is partially defined.
Trivially, the identity, constant and projection functions are simulatable,
and consequently,
by the above lemma any feed-forward circuit (network) of simulatable partial functions is again simulatable.\footnote{Although this is essentially all we need in this article to prove our main result, the above setup can be extended to more complicated combinations as {\it recursion} in classical recursion theory and theory of computation.  We do not discuss the general setup in this article (e.g. see \cite{DAN01} for an example of this). Also, all the results of this article can be generalized to the case of effective simulation of {\it relations} with appropriate changes made in definitions and sequential coloring algorithms.  However we do not delve into the details of these generalizations since it will not give rise to stronger results as far as the main result of this article is concerned.}
\subsection{Graph amalgams}
In what follows we recall a more or less standard notation for
amalgams which is adopted from \cite{DHT?}
(the interested reader may also consult \cite{DE08,DHT?,ROGR97,EEKR99-1,EEKR99-2} for more
on this as well as the Appendix of this article).\\
Let $X=\{x_{_{1}},x_{_{2}},\ldots,x_{_{k}}\}$ and $\matr{G}$ be a
set and a graph, respectively, and consider a one-to-one map
$\varrho: X \longrightarrow V(\matr{G})$. Evidently, one can
consider $\varrho$ as a graph monomorphism from the empty graph
$\matr{X}$ on the vertex set $X$ to the graph $\matr{G}$.
We interpret it as a {\it labeling} of some
vertices of $\matr{G}$ by the elements of $X$. The data introduced
by $(X,\matr{G},\varrho)$ is called a {\it marked graph} $\matr{G}$
marked by the set $X$ through the map $\varrho$. Note that (by abuse
of language) we may introduce the corresponding marked graph as
$\matr{G}[x_{_{1}},x_{_{2}},\ldots,x_{_{k}}]$ when the definition of
$\varrho$ (especially its range) is clear from the context. Also,
 we may refer to {\it the vertex $x_{_{i}}$}
as the vertex $\varrho(x_{_{i}}) \in V(\matr{G})$. This is most
natural when $X \subseteq V(G)$ and vertices in $X$ are marked by
the corresponding elements in
$V(G)$ through the identity mapping.
We use the following notation
$$\matr{G}[x_{_{1}},x_{_{2}},\ldots,x_{_{k}}]+\matr{H}[y_{_{1}},y_{_{2}},\ldots,y_{_{l}}]$$
for the graph (informally) constructed by taking the disjoint union
of the marked graphs $\matr{G}[x_{_{1}},x_{_{2}},\ldots,x_{_{k}}]$
and $\matr{H}[y_{_{1}},y_{_{2}},\ldots,y_{_{l}}]$ and then {\it
identifying} the vertices with the same mark (for a formal and
precise definition see the Appendix). It is understood that the
repeated appearances of a graph structure in an expression such as
$\matr{G}[v]+\matr{G}[v,w]$ are always made of disjoint copies of
that structure with the indicated labels marked and identified properly.
For example, $\matr{G}[v]+\matr{G}[v,w]$ is an amalgam constructed
by two disjoint isomorphic copies of $\matr{G}$ identified on the
vertex $v$ where the vertex $w$ in one of these copies is marked.
In the sequel, we may also use a semicolon to emphasize or separate two lists of vertices in a marked graph, as $\matr{G}[x,y;z]$.
Also,  we may use a bold vector notation (or a set of vertices when it causes no confusion) to refer to a list of marked vertices as $\matr{G}[\vect{v}] \isdef \matr{G}[v_{_{1}},v_{_{2}},\ldots,v_{_{k}}]$. \\
By $\matr{K}_{_{k}}[v_{_{1}},v_{_{2}},\ldots,v_{_{k}}]$ we mean a
$k$-clique on $\{v_{_{1}},v_{_{2}},\ldots,v_{_{k}}\}$ marked by its
own set of vertices. A single edge is denoted by
$\eg[v_{_{1}},v_{_{2}}]$ (i.e.,
$\eg[v_{_{1}},v_{_{2}}] =
\matr{K}_{_{2}}[v_{_{1}},v_{_{2}}]$). As one more simple example
note that
$$\matr{P}_{_{2}}[v_{_{1}},v_{_{2}},v_{_{3}}] \isdef \eg[v_{_{1}},v_{_{2}}]+\eg[v_{_{2}},v_{_{3}}]$$
is a path of length $2$ on the vertex set
$\{v_{_{1}},v_{_{2}},v_{_{3}}\}$.\\ \ \\
\section{A simulation lemma for invertible functions}\label{sec:invertiblefuncs}
Our main objective in this section is to prove that any invertible function as $\pi:S\rightarrow S$ can be simulated through graph colorings. For this we need the following basic lemma.
\begin{lem}\label{lem:Lfunction}
Given $n \geq 3$, for any fixed $k \in C$,  there  exists
 a graph $\Lkn$ that satisfies the following properties,
\begin{itemize}
\item[{\rm i)}]{Any partial proper $n$-coloring $\sigma$ for which $\sigma(x)=k$ or $\sigma(y)=k$ uniquely extends to an $n$-coloring of
 $\Lkn$. $($This also implies that in any $n$-coloring of $\Lkn$, the vertex $x$ takes the color
$k$ if and only if the vertex $y$ takes the color $k$.$)$}
\item[{\rm ii)}]{Any assignment of colors from $C_{_{\{k\}}}$ to
vertices $x$ and $y$ has a unique extension to an $n$-coloring of
the whole graph $\Lkn$.}
\item[{\rm iii)}]{For the graph $\Lkn$ the number of vertices is equal to $4n-3$ and the number of edges is equal to $\frac{n(7n-11)}{2}$. }
\end{itemize}
\end{lem}
\begin{proof}{
We define the graphs $\eta_{_{k}}[x,\vect{u}_{_{\{k\}}},y;R]$ and
 $\zeta_{_{k}}(i,j)[a,b;R]$ (depicted in Figure~\ref{fig:gzt}$(a,b)$)
by
$$\eta_{_{k}}[x,\vect{u}_{_{\{k\}}},y;R] \isdef
     \sum_{i \in C_{_{\{k\}}}}\matr{P}_{_{2}}[x,u_{_{i}},y] + \sum_{i \in{C_{_{\{k\}}} }}~
     \sum_{  z \in R_{_{\{\Kv{k},\Kv{i}\}}}}\eg[u_{_{i}},z]+\matr{K}_{_{n}}[R]$$
 and
 $$\zeta_{_{k}}(i,j)[a,b;R] \isdef
     \matr{P}_{_2}[a,v,b]+\matr{P}_{_2}[a,w,b]+\sum_{z \in R_{_{\{\Kv{k},\Kv{i}\}}}}\eg[v,z] +
     \sum_{z \in R_{_{\{\Kv{k},\Kv{j}\}}}}\eg[w,z]+\matr{K}_{_{n}}[R].$$
Also, let $\matr{T}[\vect{u}_{_{\{k\}}}]$ be an arbitrary tree on
$n-1$ vertices marked by the elements of $\vect{u}_{_{\{k\}}}$ (see
Figure~\ref{fig:gzt}$(c)$). Define
$\matr{T}_{_{\zeta}}[\vect{u}_{_{\{k\}}};R]$ to be the graph
constructed on $\matr{T}[\vect{u}_{_{\{k\}}}]$ by substituting each
edge $\eg[u_{_{i}},u_{_{j}}]$ by the structure
$\zeta_{_{k}}(i,j)[u_{_{i}},u_{_{j}};R]$; that is,
$$\matr{T}_{_{\zeta}}[\vect{u}_{_{\{k\}}};R] \isdef \displaystyle{\sum_{u_{_{i}}u_{_{j}} \in E(\matr{T}[\vect{u}_{_{\{k\}}}])}} \zeta_{_{k}}(i,j)[u_{_{i}},u_{_{j}};R].$$
Now, define
$$\Lkn \isdef
\eta_{_{k}}[x,\vect{u}_{_{\{k\}}},y;R]+\matr{T}_{_{\zeta}}[\vect{u}_{_{\{k\}}};R]. $$
To prove (i) note that if for a partial $C$-coloring $\sigma$ of $\Lkn$ we have $\sigma(x)=k$,
then just by restricting ourselves to $\eta_{_{k}}[x,\vect{u}_{_{\{k\}}},y]$ we have $\sigma(u_{_{i}})=i$ for all $i \in C_{\{k\}}$, and consequently,
$\sigma(y)=k$. Now, it is easy to see that this partial coloring uniquely extends to a $C$-coloring of $\Lkn$. The rest of the proof follows by symmetry.\\
To prove (ii), first we prove the following claim.
\begin{itemize}
\item[]{{\bf claim}: {\it In any $C$-coloring $\sigma$ of $\Lkn$, if there is an index $i \in C$ such that $\sigma(u_{_{i}})=k$ then we have
$$\forall\ j \in C_{_{\{k\}}} \quad \sigma(u_{_{j}})=k.$$}
It is easy to see that if $u_{_{l}}u_{_{j}} \in E(\matr{T}[\vect{u}_{_{\{k\}}}])$ and $\sigma(u_{_{l}})=k$, then by the structure of
$\zeta_{_{k}}(l,j)[u_{_{l}},u_{_{j}}]$ we have $\sigma(u_{_{j}})=k$. Then, the claim follows from the connectedness of the tree $\matr{T}[\vect{u}_{_{\{k\}}}]$ and the fact that
for any $j \not = i$ there is a path in this tree starting from the vertex $u_{_{i}}$ and ending at the vertex $u_{_{j}}$.
}
\end{itemize}
Now, to prove (ii), let  the vertices  $x$ and $y$ be given
arbitrarily colors $i$ and $j$ in $C_{_{\{k\}}}$, respectively.
Then, by the structure of $\eta_{_{k}}[x,\vect{u}_{_{\{k\}}},y]$ it is clear that the vertex $u_{_{i}}$ is forced to take the color $k$,
and consequently, by the previously proved claim all vertices $u_{_{l}}$ are forced to take the color $k$. This again fixes the color of the rest
of the vertices appearing in the structures $\zeta_{_{k}}(l_{_{1}},l_{_{2}})[u_{_{l_{_{1}}}},u_{_{l_{_{2}}}}]$.\\
Part (iii) is straight forward and can be verified using the definition of the graph $\Lkn$.
 }\end{proof}
\begin{figure}[!htbp]
      \begin{center}
         $$\diagEta$$
         (a) $\eta_{_{k}}[x,\vect{u}_{_{\{k\}}},y;R]$
      \end{center}
   \begin{minipage}[b]{.49\textwidth}
      \begin{center}
        $$\diagzeta$$
        (b) $\zeta_{_{k}}(i,j)[a,b;R]$
      \end{center}
   \end{minipage}
   \begin{minipage}[b]{.49\textwidth}
      \begin{center}
        $$\diagtree$$
        (c) $\matr{T}_{_{{\zeta}}}[\vect{u}_{_{\{k\}}};R]$
      \end{center}
   \end{minipage} \\
   \caption{Components of $\Lkn$}
   \label{fig:gzt}
\end{figure}
\begin{figure}
      \begin{center}
        $$\diagfanar$$
   \caption{$\Lkn$}
    \label{fig:diagfanar}
     \end{center}
\end{figure}
\begin{pro}\label{pro:permutations}
Given integers $ m \geq 2$ and $n \geq \max\{m,3\}$, along with a set
$$S = \{0,1,\ldots,m-1\}$$
and an invertible function $\pi:S\rightarrow S$ that can be presented using $\Theta$ transpositions, then there is a graph
$\matr{G}_{_{\pi,n}}[x;y]$ that simulates $\pi$ through
$n$-colorings. Moreover, $|V(\matr{G}_{_{\pi,n}})| \simeq O(\Theta n^2)$ and $|E(\matr{G}_{_{\pi,n}})|\simeq O(\Theta n^3)$.
\end{pro}
\begin{figure}[ht]
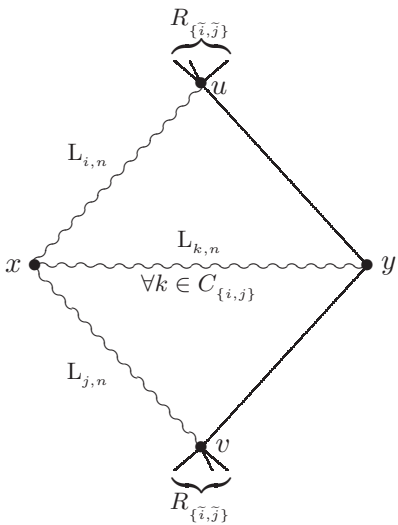

   \begin{center}
      $$\diagT$$
   \end{center}
   \caption{$\matr{G}_{_{\tau(i,j),n}}[x;y;R]$}
   \label{fig:T}
\end{figure}
\begin{proof}{ Let $\tau(i,j)$ be a transposition on $S$. We construct the graph
$\matr{G}_{_{\tau(i,j),n}}[x;y;R]$  as follows (see Figure~\ref{fig:T}):
\begin{equation}\label{eq:transposition}
\begin{array}{ll}
\matr{G}_{_{\tau(i,j),n}}[x;y;R]& \isdef
\displaystyle \sum_{k \in \matr{C}_{_{\{i,j\}}}} \Lkn
+\sum_{z\in{R_{_{\{\Kv{i},\Kv{j}\}}}}}\eg[u,z]\\
&\\
&+\displaystyle \sum_{z\in{R_{_{\{\Kv{i},\Kv{j}\}}}}}\eg[v,z]+\sum_{m \leq k < n}\eg[x,\Kv{k}]+\sum_{m \leq k < n}\eg[y,\Kv{k}]\\
&\\
&+\matr{L}_{_{i,n}}[x,u;R]+\matr{L}_{_{j,n}}[x,v;R]
+\eg[u,y]+\eg[v,y].
\end{array}
\end{equation}
We show that the graph $\matr{G}_{_{\tau(i,j),n}}[x;y]$
simulates the transposition $\tau(i,j)$ through
$n$-colorings. For this we consider the following two cases.
\begin{itemize}
\item{
If the vertex $x$ takes the color $i\in S$, then by
Lemma~\ref{lem:Lfunction}(i) the vertex $u$ is also forced to
take the colored $i$. Also, since $x$  is connected to $v$ by the
structure $\matr{L}_{_{j,n}}[x,v]$, the vertex $v$ can not  be colored by the color $j$.\\
On the other hand, note that since $x$  is connected to $y$ by the structures
$\matr{L}_{_{k,n}}[x,y]$ for all $k \in C_{_{\{i,j\}}}$, the vertex
$y$ can only take the color $j$.
The same proof goes through by symmetry when $x$ takes the color $j$.
}
\item{
If the vertex $x$ takes a color $k \in S_{_{\{i,j\}}}$, then, the vertex $y$ is forced to take the color $k$,
since $x$ is connected to $y$ through the structure $\matr{L}_{_{k,n}}[x,y]$.
 Consequently, the vertices $u$ and $v$ take the colors $j$ and $i$, respectively.
}
\end{itemize}
The number of vertices and edges can be verified directly.
Now, for a given permutation $\pi$ that can be presented by $\Theta$ transpositions, the proposition follows from Lemma~\ref{lem:combination}.
 }\end{proof}
As a corollary, let us   define the functions $\Cs^{+}_{_{m}}$ and $\Cs^{-}_{_{m}}$ on $S$ as
$$\Cs^{+}_{_{m}}(i) \isdef i+1 \pmod{m}\quad\text{and}\quad \Cs^{-}_{_{m}}(i) \isdef i-1 \pmod{m}.$$
Then, by Proposition~\ref{pro:permutations}
can be simulated by graph colorings. Also, we have
$$|V(\matr{G}_{_{\Cs^{+}_{_{m}},n}})| \simeq |V(\matr{G}_{_{\Cs^{-}_{_{m}},n}})| \simeq O(n^3),$$
and
$$|E(\matr{G}_{_{\Cs^{+}_{_{m}},n}})| \simeq |E(\matr{G}_{_{\Cs^{-}_{_{m}},n}})| \simeq O(n^4).$$
We will use these functions in our forthcoming constructions.
\section{A couple of gadgets}\label{sec:gadgets}
\subsection{Some simple gadgets }
In this section we prove that a couple of basic gadgets can be simulated through graph colorings. These gadgets will be used in our
final construction.
\begin{defin}{
For any $m \geq 2$, we define the following functions.
$$\begin{array}{ll}
\Ch_{_{m}} : S \rightarrow \{0,1\}^{m}\;, &\quad    \Ch_{_{m}}(i) \isdef \vect{1}_{_{m}}^{^i} \;.\\
&\\
\Ps_{_{m}} : \{0,1\}^{m} \rightarrow S\;, &\quad  \Ps_{_{m}}(\vect{u}) \isdef \left \{ \begin{array}{ll}
                                                                  i  & \text{if $\vect{u}=\vect{1}_{_{m}}^{^i}$,}  \\
                                                                  {\rm undefined} & \text{otherwise}. \end{array}\right.\\
&\\
\Ex_{_{m}} : S \rightarrow \{0,1\}^{m}\;,&\quad  \Ex_{_{m}}(i)\isdef \left \{ \begin{array}{ll}
                                                                    \vect{0}_{_{m}} &  \text{if $i=0$,} \\
                                                                    \vect{1}_{_{m}} &  \text{if $i=1$,}\\
                                                                    {\rm undefined} & \text{otherwise.} \end{array}\right.
\end{array}
$$
Note that $\Ps_{_{m}}$ and $\Ch_{_{m}}$ are inverse to each other and $\Ps_{_{m}} \circ \Ch_{_{m}}=\Id_{_{m}}$.
 }\end{defin}

\begin{lem}
For any pair of integers $ m \geq 2$ and $n \geq \max\{m,3\}$,
\begin{itemize}
\item[{\rm i)}]{There are graphs $\matr{G}_{_{\Ch_{_{m}},n}}$ and $\matr{G}_{_{\Ps_{_{m}},n}}$ that simulate the functions
$\Ch_{_{m}}$ and $\Ps_{_{m}}$ through $n$-colorings, respectively.}
\item[{\rm ii)}]{There exist a graph $\matr{G}_{_{\Ex_{_{m}},n}}$ that simulates $\Ex_{_{m}}$ through  $n$-colorings.}
\item[{\rm iii)}]{We have,
$$\begin{array}{llll}
|V(\matr{G}_{_{\Ch_{_{m}},n}})| &\simeq |V(\matr{G}_{_{\Ps_{_{m}},n}})|& \simeq |V(\matr{G}_{_{\Ex_{_{m}},n}})|&  \simeq O(n^2), \\
|E(\matr{G}_{_{\Ch_{_{m}},n}})| &\simeq |E(\matr{G}_{_{\Ps_{_{m}},n}})|& \simeq |E(\matr{G}_{_{\Ex_{_{m}},n}})|&  \simeq O(n^3).
\end{array}
$$}
\end{itemize}
\end{lem}
\begin{figure}[t]
      \begin{center}
        $$\diagchar$$
        \caption{$\matr{G}'_{_{\Ch_{_{m}},n}}[x;\vect{u};R]$}
        \label{fig:diagchar}
      \end{center}
\end{figure}
\begin{proof}{Firstly, we prove (i). For this let $\vect{u}=(u_{_{0}},u_{_{1}},\ldots,u_{_{m-1}})$, and define the
graph $\matr{G}_{_{\Ch_{_{m}},n}}[x;\vect{u};R]$ by (see Figure~\ref{fig:diagchar})
$$\begin{array}{ll}
\matr{G}'_{_{\Ch_{_{m}},n}}[x;\vect{u};R] &\isdef
 \eg[w_{_{0}},u_{_{0}}]+\displaystyle \sum_{i=0}^{m-1}\matr{L}_{_{i,n}}[x,w_{_{i}};R]+
  \sum_{i=1}^{m-1}\matr{L}_{_{0,n}}[w_{_{i}},u_{_{i}};R]
  \\ \\
&
 \ \ +\displaystyle \sum_{i=2}^{n-1}\eg[w_{_{0}},\Kv{i}]+
 \sum_{j=1}^{m-1}\sum_{i \in C_{_{\{0,j\}}}}\eg[w_{_{j}},\Kv{i}]+
 \sum_{j=0}^{m-1}\sum_{i=2}^{n-1}\eg[u_{_{j}},\Kv{i}] \;.
 \end{array}$$
Suppose that $x$ is assigned a color $k$, where $0\leq k<m$,
and each $\widetilde{i}\in R$ is assigned the color $i$.
Then, it follows from the property of the graph $\matr{L}$
(Lemma~\ref{lem:Lfunction}) that this assignment can be extended
to a proper $n$-coloring in which $w_{_k}$ takes the color $k$,
$u_{_k}$ the color $1$, and all the vertices $w_{_j}$ and $u_{_j}$
for $j\neq k$ the color $0$, and furthermore, this extension is unique.
Therefore, the graphs
$$\matr{G}_{_{\Ch_{_{m}},n}}[x;\vect{u};R] \isdef
\matr{G}'_{_{\Ch_{_{m}},n}}[x;\vect{u};R] +
\sum_{i=m}^{n-1} \eg[x,\Kv{i}]$$
and
$$\matr{G}_{_{\Ps_{_{m}},n}}[\vect{u};x;R] \isdef
\matr{G}'_{_{\Ch_{_{m}},n}}[x;\vect{u};R] +
\sum_{i=m}^{n-1} \eg[x,\Kv{i}]$$
simulate $\Ch_{_{m}}$ and $\Ps_{_{m}}$, respectively.\\
Secondly, we prove (ii). For this we define
$$
\begin{array}{ll}
\matr{G}_{_{\Ex_{_{m}},n}}[x;\vect{v};R]& \isdef \displaystyle{\sum_{i=0}^{m-1}}(\matr{L}_{_{0,n}}[x,v_{_{i}};R]+\matr{L}_{_{1,n}}[x,v_{_{i}};R]) \\
&\\
&\ + \displaystyle{\sum_{j=0}^{m-1}\sum_{i=2}^{n-1}}\ \eg[v_{_{j}},\Kv{i}].
\end{array}
$$
Similar to the previous case, it is quite straight forward to verify that $\matr{G}_{_{\Ex_{_{m}},n}}[x;\vect{v};R]$ simulates the partial function $\Ex_{_{m}}$.
}\end{proof}
Next, we define the extended basic Boolean partial operations as follows.
\begin{defin}{
$$\begin{array}{ll}
\Oand_{_{m}}: S \times S \rightarrow \{0,1\}\;, &\quad  \Oand_{_{m}}(x,y) \isdef \left \{ \begin{array}{ll}
                                                                           x \wedge y &   \text{if $x,y\in\{0,1\}$,} \\
                                                                           {\rm undefined} & \text{otherwise.} \end{array}\right.\\
&\\
 \Oor_{_{m}}: S \times S \rightarrow \{0,1\}\;, &\quad  \Oor_{_{m}}(x,y) \isdef \left \{ \begin{array}{ll}
                                                                           x \vee y & \text{if $x,y\in\{0,1\}$,} \\
                                                                           {\rm undefined} & \text{otherwise.} \end{array}\right.\\
&\\
\Onot_{_{m}}: S \rightarrow \{0,1\}\;, &\quad  \Onot_{_{m}}(x) \isdef \left \{ \begin{array}{ll}
                                                                           \neg x & \text{if $x\in\{0,1\}$,} \\
                                                                           {\rm undefined} & \text{otherwise.} \end{array}\right.
\end{array}
$$
}\end{defin}
The space of {\it extended Boolean partial functions} consists of all functions that can be constructed using a finite combination of extended basic Boolean partial operations. The next lemma is an extension of a result of \cite{DHT?} for the extended Boolean partial functions.
\begin{figure}[t]
   \begin{center}
      $$\diagand$$
   \end{center}
   \caption{$\matr{G}_{_{\Oand_{_{m}},n}}[x,y;z;R]$}
   \label{fig:and}
\end{figure}
\begin{lem}\label{lem:boolean}
Any extended Boolean partial function  on $S=\{0,1,\ldots,m-1\}$ can be
simulated through $n$-colorings for any $n \geq \max\{m,3\}$. Moreover,
$$\begin{array}{lll}
|V(\matr{G}_{_{\Oand_{_{m}},n}})| &\simeq |V(\matr{G}_{_{\Oor_{_{m}},n}})|& \simeq O(n^2), \\
|E(\matr{G}_{_{\Oand_{_{m}},n}})| &\simeq |E(\matr{G}_{_{\Oor_{_{m}},n}})|& \simeq O(n^3),\\
|V(\matr{G}_{_{\Onot_{_{m}},n}})| & \simeq O(n),& \\
|E(\matr{G}_{_{\Onot_{_{m}},n}})| & \simeq O(n^2).&
\end{array}
$$
\end{lem}
\begin{proof}{By Lemma~\ref{lem:combination} it suffices to prove the result for the basic Boolean partial operations $\Oand_{_{m}}$ and $\Onot_{_{m}}$.
It is easy to check that the graph $\matr{G}_{_{\Oand_{_{m}},n}}$  defined as (see Figure~\ref{fig:and})
$$
\begin{array}{ll}
\matr{G}_{_{\Oand_{_{m}},n}}[x,y;z;R] & \isdef \matr{L}_{_{0,n}}[x,w;R] + \matr{G}_{_{\Cs^{^+}_{_{m}},n}}[y,v;R]+\matr{L}_{_{1,n}}[w,z;R]+\eg[v,w]\\
&\\
&\ + \displaystyle{\sum_{i=2}^{n-1}}\eg[x,\Kv{i}]+\displaystyle{\sum_{i=2}^{n-1}}\eg[y,\Kv{i}]+
\displaystyle{\sum_{i=3}^{n-1}}\eg[w,\Kv{i}]+\displaystyle{\sum_{i=2}^{n-1}}\eg[z,\Kv{i}]
\end{array}
$$
simulates the $\Oand_{_{m}}$ Boolean partial operation. Also, similarly, one may verify that the graph
$$
\begin{array}{ll}
\matr{G}_{_{\Onot_{_{m}},n}}[x;y;R] & \isdef \eg[x,y]+\matr{K}_{_{n}}[R]
+\displaystyle{\sum_{i=2}^{n-1}}\eg[x,\Kv{i}]
+\displaystyle{\sum_{i=2}^{n-1}}\eg[y,\Kv{i}]
\end{array}
$$
simulates the $\Onot_{_{m}}$ Boolean function.
}\end{proof}
\subsection{An edge-simulation gadget}
Our main objective in this section is to prove that the following partial function
$$\begin{array}{ll}
\Edg_{_r}: S^{r} \times S^{r} \rightarrow S^{r}\;, &\quad \Edg_{_r}(\vect{u},\vect{v}) \isdef  \left \{ \begin{array}{ll}
                                                                                      \vect{v} & \text{if $\vect{u} \not = \vect{v}$,}\\
                                                                                     {\rm undefined} & \text{if $\vect{u} = \vect{v}$,} \end{array}\right.
\end{array}
$$
can be simulated through $n$-colorings. Note that the coloring behavior of $\Edg_{_r}(\vect{u},\vect{v})$ with respect to $r$-vectors $\vect{u}$ and $\vect{v}$ is like an edge. In this regard, we will need a control gadget that behaves very similar to a binary multiplexor:
\begin{defin}{Let us define,
$${\bf U}_{m} \isdef \{\vect{1}_{_{m}}^{^i} \ \ | \ \ 0 \leq i \leq m-1\}\;.$$
Then the $\Mux_{_m}$ function is defines as,
$$\begin{array}{ll}
\Mux_{_m}: S^{m} \times {\bf U}_{m} \rightarrow S\;, &\quad \Mux_{_m}(\vect{v},\vect{1}_{_{m}}^{^i}) \isdef  v_{_i}\;.
\end{array}
$$
}\end{defin}
In the next lemma we prove that $\Mux_{_m}$ can be simulated through $n$-colorings. The main idea behind the proof is to reduce
the operations to the Boolean level and then again extend to the $S$-level using the functions introduced in the previous section.
\begin{lem}{
For any pair of integers $m \geq 2$ and $ n \geq \max\{m,3\}$, there is a graph
$\matr{G}_{_{\Mux_{_m}},n}[\vect{v},\vect{u};t;R]$ that
simulates the function $\Mux_{_{m}}$ through $n$-colorings. Moreover,
$$|V(\matr{G}_{_{\Mux_{_m},n}})| \simeq  O(n^4) \quad {\rm and} \quad |E(\matr{G}_{_{\Mux_{_m},n}})| \simeq  O(n^5).$$
}\end{lem}
\begin{figure}[t]
\special{em:linewidth 1pt} \unitlength 0.25mm \linethickness{0.4pt}
  \begin{center}
        $$\diagD$$
             \end{center}
\caption{\protect\label{ucg3} $\matr{G}_{_{\Dot_{_{m}},n}}[x,y;z;R]$ }
\label{fig:D}
\end{figure}
\begin{figure}[t]
\special{em:linewidth 1pt} \unitlength 0.25mm \linethickness{0.4pt}
  \begin{center}
        $$\diagP$$
         \end{center}
\caption{\protect\label{ucg3}
$\matr{G}_{_{\Ext_{_m},n}}[\vect{x};t;R]$} \label{fig:P}
\end{figure}
\begin{proof}{In the proof we need the following functions and their coloring simulations.
$$\begin{array}{ll}
\Dot_{_{m}}: S\times \{0,1\}   \rightarrow S\;, &\quad \Dot_{_{m}}(i,j)\isdef \left \{ \begin{array}{ll}
                                                                 0 & \text{if $j=0$,} \\
                                                                 i & \text{if $j=1$,} \end{array}\right.\\
&\\
\Ext_{_m}: S^{m} \rightarrow S\;, & \quad \Ext_{_m}(\vect u)\isdef \left \{ \begin{array}{ll}
                                                                 k & \text{if $\forall\ i \ \ u_{_{i}} \in \{0,k\}$,} \\
                                                                 {\rm undefined}  &  \text{otherwise.} \end{array}\right.
\end{array}
$$
\begin{itemize}
\item[]{{\bf claim}$\ 1.$ {\it Using Lemma~{\rm\ref{lem:combination}}, it is straight forward to check that the graph $\matr{G}_{_{\Dot_{_{m}},n}}[x,y;z;R]$
defined as follows $($see Figure~{\rm \ref{fig:D}$)$}
simulates the function $\Dot_{_{m}}$ through $n$-colorings. Moreover,
$$|V(\matr{G}_{_{\Dot_{_{m}},n}})| \simeq  O(n^3) \quad {\rm and} \quad |E(\matr{G}_{_{\Dot_{_{m}},n}})| \simeq  O(n^4).$$
}\ \\
$$
\begin{array}{ll}
\matr{G}_{_{\Dot_{_{m}},n}}[x,y;z;R] & \isdef
\matr{G}_{_{\Ch_{_{m}},n}}[x;\vect{v};R]+\matr{G}_{_{\Ex_{_{m}},n}}[y;\vect{u};R]+\matr{G}_{_{\Ps_{_m},n}}[\vect{w};z;R]\\
&\\
&+\matr{G}_{_{\Oor_{_{m}},n}}[v_{_{0}},t;w_{_{0}};R]+\matr{G}_{_{\Onot_{_{m}},n}}[u_{_{0}};t;R]\\
&\\
&+\displaystyle{\sum_{k=1}^{m-1}}\matr{G}_{_{\Oand_{_{m}},n}}[v_{_{k}},u_{_{k}};w_{_{k}};R].
\end{array}
$$
}
\item[]{{\bf claim}$\ 2.$ {\it Using  Lemma~{\rm \ref{lem:combination}}, it is straight forward to check that the graph $\matr{G}_{_{\Ext_{_m},n}}[\vect{x};t;R]$ defined as follows $($see  Figure~{\rm \ref{fig:P}}$)$
     simulates the function $\Ext_{_{m}}$ through $n$-colorings. Moreover,
$$|V(\matr{G}_{_{\Ext_{_{m}},n}})| \simeq  O(n^4) \quad {\rm and} \quad |E(\matr{G}_{_{\Ext_{_{m}},n}})| \simeq  O(n^5).$$
 }\\
First, note that for $r\geq 2$ and $\vect{x}=(x_{_{0}},x_{_{1}},\ldots ,x_{_{r-1}})$, by  Lemma~{\rm \ref{lem:boolean}}, the following functions can be simulated through $n$-colorings:
$$\begin{array}{ll}
 \Vor_{_{r}}: \{0,1\}^{^r}\rightarrow\{0,1\}\;, &\quad  \Vor_{_{r}}(\vect{x}) \isdef \left \{ \begin{array}{ll}
                                                                           0 & \text{if $\forall\ i \quad  x_{_{i}}=0$,} \\
                                                                           1 & \text{otherwise,} \end{array}\right.\\
&\\
\Vand_{_{r}}: \{0,1\}^{^r}\rightarrow\{0,1\}\;, &\quad  \Vand_{_{r}}(\vect{x})\isdef \left \{ \begin{array}{ll}
                                                                            0 & \text{if $\exists\ i \quad  x_{_{i}}=0$,}\\
                                                                            1 & \text{otherwise.} \end{array}\right.\\
\end{array}
$$
Then define
$$\begin{array}{ll}
 \matr{G}_{_{\Ext_{_m},n}}[\vect{x};t;R] &\isdef  \matr{G}_{_{\Vand_{_m},n}}[v_{_{0}}^{0},v_{_{0}}^{1},\ldots,v_{_{0}}^{{m-1}};w_{_{0}};R]\\
 &\\
   &+\matr{G}_{_{\Ps_{_m},n}}[\vect{w};t;R]+\displaystyle{\sum_{i=0}^{m-1}}\matr{G}_{_{\Ch_{_{m}},n}}[x_{_{i}};\vect{v}^{i};R]\\
&\\
&+\displaystyle{\sum_{i=1}^{m-1}}\matr{G}_{_{\Vor_{_m},n}}[v_{_{i}}^{0},v_{_{i}}^{1},\ldots, v_{_{i}}^{{m-1}};w_{_{i}};R].
\end{array}
$$
}
\end{itemize}
Now, it can be verified that the graph defined as (see Figure~\ref{fig:Mux})
$$
\begin{array}{ll}
\matr{G}_{_{\Mux_{_m}},n}[\vect{v},\vect{u};t;R] & \isdef \displaystyle{\sum_{i=0}^{m-1}}{\matr{G}_{_{\Dot_{_{m}},n}}}[v_{_{i}},u_{_{i}};w_{_{i}};R]+\matr{G}_{_{\Ext_{_m},n}}[\vect{w};t;R],
\end{array}
$$
simulates the multiplexor function through $n$-colorings.
}\end{proof}
\begin{figure}[t]
   \special{em:linewidth 1pt} \unitlength 0.25mm \linethickness{0.4pt}
   \begin{center}
        $$\diagMux$$
           \end{center}
   \caption{ \protect\label{ucg3} $\matr{G}_{_{\Mux_{_m}},n}[\vect{v},\vect{u};t;R]$  }
   \label{fig:Mux}
\end{figure}
Finally,  we focus on the basic arithmetic operations.
\begin{defin}{
$$\begin{array}{ll}
 \Add_{_{m}}: S \times S \rightarrow S\;, &\quad  \Add_{_{m}} (x,y) \isdef x+y \pmod{m}\;,\\
 \Sub_{_{m}}: S \times S \rightarrow S\;, &\quad  \Sub_{_{m}} (x,y) \isdef x-y \pmod{m}\;.
\end{array}
$$
}\end{defin}
\begin{lem}\label{lem:arthmetic}
For any pair of integers $m \geq 2$ and $ n \geq \max\{m,3\}$,
there are graphs $\matr{G}_{_{\Add_{_{m}},n}}$ and $\matr{G}_{_{\Sub_{_{m}},n}}$ that simulate the functions
$\Add_{_{m}}$ and $\Sub_{_{m}}$ through $n$-colorings, respectively. Moreover,
$$\begin{array}{lll}
|V(\matr{G}_{_{\Add_{_{m}},n}})| &\simeq |V(\matr{G}_{_{\Sub_{_{m}},n}})|& \simeq O(n^4), \\
|E(\matr{G}_{_{\Add_{_{m}},n}})| &\simeq |E(\matr{G}_{_{\Sub_{_{m}},n}})|& \simeq O(n^5).
\end{array}
$$
\end{lem}
\begin{figure}[t]
   \special{em:linewidth 1pt} \unitlength 0.25mm \linethickness{0.4pt}
   \begin{center}
        $$\diagDel$$
           \end{center}
   \caption{\protect\label{ucg3} $\matr{G}_{_{F,n}}[v_{_{0}},y;t;R]$ }
   \label{fig:Del}
\end{figure}
\begin{proof}{It is easy to see that using the setup of Figure~\ref{fig:Del}
and using the functions $\Cs^{+}_{m}$ and $\Cs^{-}_{m}$ for the black-box $F$ one gets the following graphs
$$\begin{array}{ll}
\matr{G}_{_{\Add_{_{m}},n}}[v_{_{0}},y;t;R]  &\isdef  \matr{G}_{_{\Ch_{_{m}},n}}[y;\vect{u};R]
+\matr{G}_{_{\Mux_{_m}},n}[\vect{v},\vect{u};t;R]\\
&\\
&+\displaystyle{\sum_{i=0}^{m-2}}\matr{G}_{_{\Cs^{+}_{_m},n}}[v_{_{i}};v_{_{i+1}};R],\\
&\\
\matr{G}_{_{\Sub_{_{m}},n}}[v_{_{0}},y;t;R]  &\isdef  \matr{G}_{_{\Ch_{_{m}},m}}[y;\vect{u};R]
+\matr{G}_{_{\Mux_{_m}},n}[\vect{v},\vect{u};t;R]\\
&\\
&+\displaystyle{\sum_{i=0}^{m-2}}\matr{G}_{_{\Cs^{-}_{_m},n}}[v_{_{i}};v_{_{i+1}};R],
\end{array}
$$
that simulate the $\Add_{_{m}}$ and $\Sub_{_{m}}$ functions.\\
(It is worth noting that by Proposition~\ref{pro:permutations},
the inversion $i\mapsto -i \pmod{m}$ can also be simulated through $n$-colorings,
and consequently, one can also construct the simulator of $\Sub_{_{m}}$ by a composition of $\Add_{_{m}}$ and inversion.  However, this method will give
rise to an increase in the size of the corresponding graph.)
}\end{proof}
Now, we are ready to prove the main result of this section.
\begin{pro}\label{pro:edgesimulation}
For any integers $m \geq 2$ and $ n \geq \max\{m,3\}$, $r \geq 1$ and $S = \{0,1,\ldots,m-1\}$, there is a graph
$\matr{E}_{_{r,n}}[\vect{u},\vect{v};\vect{v};R]$ that
simulates the function $\Edg_{_{r}}: S^{r} \times S^{r} \rightarrow S^{r}$ through $n$-colorings.
Moreover,
$$|V(\matr{E}_{_{r,n}})| \simeq  O(rn^4) \quad {\rm and} \quad |E(\matr{E}_{_{r,n}})| \simeq  O(rn^5).$$
\end{pro}
\begin{proof}{If $r=1$ then define
$\matr{E}_{_{1,n}}[u,v;v;R] \isdef e[u,v].$
If $r > 1$ then define  $\matr{E}_{_{r,n}}[\vect{u},\vect{v};\vect{v};R]$ as
$$\begin{array}{ll}
\matr{E}_{_{r,n}}[\vect{u},\vect{v};\vect{v};R] &\isdef \displaystyle{\sum_{i=0}^{r-1}}\matr{G}_{_{\Sub_{_{m}},n}}[u_{_{i}},v_{_{i}};w_{_{i}};R]
+\displaystyle{\sum_{i=0}^{r-1}}\matr{L}_{_{0,n}}[w_{_{i}},z_{_{i}};R]\\
&\\
&+\matr{G}_{_{\Vor_{_r},n}}[z_{_{0}},z_{_{1}},\ldots,z_{_{r-1}};y;R]+\eg[y,\Kv{0}].
\end{array}
$$
and note that in any $n$-coloring, $\sigma$, the vertex $\sigma(y)=0$ if and only if for all $0 \leq i \leq r-1$ we have
$\sigma(u_{_{i}})=\sigma(v_{_{i}})$.
 }\end{proof}
\section{The main theorem}\label{sec:main}
Consider the space of functions $${\bf Fin} \isdef \{\varphi: A \to B \ \ | \ \ |A| < \infty, \ \ |B| < \infty \}.$$
 In this section we show that
${\bf Fin}$ can simulated through graph colorings.
One of the important aspects of our result is the fact that the number of colors, $n$,  can be as small as $\max\{m,3\}$, where for this and the function coding we adopt an edge simulation gadget in a cylindrical construction, along with a general idea of using an extension to invertible functions as it is usually done in quantum computing (e.g. see  \cite{NICH00}). We start with the extension lemma.
\begin{lem}
\label{lem:geninvfunc}
   Given integers $p \geq 1$, $q \geq 1$, a set $S = \{0,1,\ldots,m-1\}$ and a partial function $\varphi: S^{^p} \to S^{^q}$,
   let $\tilde{S}$ be a superset of $S$  of size $\tilde{m}$ $($i.e. {\rm w.l.g.} $S \subseteq \tilde{S} \isdef \{0,1,\ldots,\tilde{m}-1\}$ $)$ and let $r$ be an integer that satisfies the following inequality,
   $$p\log{m} \leq (r-q)\log{\tilde{m}}.$$
   Then,  for every $s_{_{0}}\in S$, there exists an invertible total function $($i.e. a permutation on $\tilde{S}^{^r})$
   $\tilde{\varphi}: \tilde{S}^{^r} \to \tilde{S}^{^r}$  such that
   $$
      \forall \ \vect{x} \in S^{^p} \quad  \tilde{\varphi}(\vect{x},s_{_{0}},\ldots,s_{_{0}})=(\varphi(\vect{x}),\vect{y}_{_{\vect{x}}}),
   $$
   for some vector $\vect{y}_{_{\vect{x}}}$ of dimension $r-q$.
\end{lem}
\begin{proof}{ The proof is essentially an straight forward application of \auth{P. Hall}'s SDR theorem (e.g. see \cite{VLWIL01}).
For more details, fix $s_{_{0}}\in S$ arbitrarily and construct $\tilde{\varphi}$ as follows.
Consider an arbitrary vector  $\vect{z}$ in $\tilde{S}^{^r}$, and define the subsets $A_{_{\vect{z}}} \subseteq \tilde{S}^{^r}$ as follows,
$$
A_{_{\vect{z}}} \isdef \left \{ \begin{array}{ll}
                           \{(\varphi(\vect{x}),\vect{y}) \ \ | \ \ \vect{y} \in \tilde{S}^{^{r-q}} \}\;, & \quad \text{if $\exists\ \vect{x} \in S^{^p} \ \  \vect{z}= (\vect{x},s_{_{0}},\ldots,s_{_{0}})$,}\\
                            & \\
                            \tilde{S}^{^r} & \quad \text{otherwise.}
                                \end{array}\right.
$$
Then the inequality in the hypothesis guarantees that
these sets satisfy \auth{P. Hall}'s SDR condition and hence there exists a system of distinct representatives $\{\vect{a}_{_{\vect{z}}}\}_{_{\vect{z} \in \tilde{S}^{^r}}}$.
Hence, one may define
$$\tilde{\varphi}(\vect{z}) \isdef \vect{a}_{_{\vect{z}}}\;.$$
}\end{proof}
The function $\tilde{\varphi}$ of Lemma~\ref{lem:geninvfunc} is called an $(\tilde{m},r,s_{_{0}})$-invertible extension of $\varphi$.
\begin{thm}\label{thrm:main}
 Given $p\geq 1$ , $q\geq 1$ , $m > 1$, $n \geq \max\{m,3\}$, a set of size $m$ $(${\rm w.l.g.}  $S = \{0,1,\ldots,m-1\})$ and a partial function $\varphi: S^{^p} \to S^{^q}$,
 there exists a graph $\matr{G}_{_{\varphi,n}}[\vect{x};\vect{y};R]$ that simulates $\varphi$ through $n$-colorings.
 Moreover,
$$|V(\matr{G}_{_{\varphi,n}})| \simeq  O(\Theta(p+q) n^{2(p+q)+4})$$
 and
$$  |E(\matr{G}_{_{\varphi,n}})| \simeq  O(\Theta(p+q) n^{3(p+q)+5}).$$
\end{thm}
\begin{proof}{
If $\varphi$ is undefined on a subset $B$ of its domain, fix an element $*$ out of the range or if necessary add a new element $*$ to the range and map $B$ to this new element. Therefore,
without loss of generality we may assume the $\varphi$ is a total function (i.e. it is everywhere defined on its domain).
Then, for $r \isdef  p+q$, we may consider the function $\tilde{\varphi}: S^{^r} \to S^{^r}$, an $(m,r,s_{_{0}})$-invertible extension of $\varphi$ given by Lemma~\ref{lem:geninvfunc}, and its simulation $\Kv{\matr{G}}_{_{\tilde{\varphi},\tilde{n}}}[x;y;R_{_{\tilde{n}}}]$,
given by Proposition~\ref{pro:permutations} for some $\tilde{n} \geq \max\{m^r,3\}$.
If $r=1$ the proof is clear by Proposition~\ref{pro:permutations}. Else, assume that $\tilde{n}=n^r$ for some $n \geq m$ and define
the simulator $\matr{G}_{_{\tilde{\varphi},n}}[\vect{x};\vect{y};R]$ as
$$\displaystyle{\sum_{wz \in E(\Kv{\matr{G}}_{_{\tilde{\varphi},\tilde{n}}})}} \
\matr{E}_{_{r,n}}[(w_{_{0}},w_{_{1}},\ldots,w_{_{r-1}}),(z_{_{0}},z_{_{1}},\ldots,z_{_{r-1}});(z_{_{0}},z_{_{1}},\ldots,z_{_{r-1}});R],$$
that can be described as the graph obtained by blowing up any vertex $w$ of $\Kv{\matr{G}}_{_{\tilde{\varphi},\tilde{n}}}$ to $r$ vertices
$w_{_{0}},w_{_{1}},\ldots,w_{_{r-1}}$ and putting a copy of $$\matr{E}_{_{r,n}}[(w_{_{0}},w_{_{1}},\ldots,w_{_{r-1}}),(z_{_{0}},z_{_{1}},\ldots,z_{_{r-1}});(z_{_{0}},z_{_{1}},\ldots,z_{_{r-1}});R]$$
if there is an edge $wz$ in $\Kv{\matr{G}}_{_{\tilde{\varphi},\tilde{n}}}$.\\
Now, note that each one of the $\tilde{n}=n^r$ colorings of a vertex $w$ of $\Kv{\matr{G}}_{_{\tilde{\varphi},\tilde{n}}}$ corresponds to a unique $n$-coloring
of the $r$-tuple $\vect{w}=(w_{_{0}},w_{_{1}},\ldots,w_{_{r-1}})$ and by Proposition~\ref{pro:edgesimulation} the behavior of the  the graph
$\matr{E}_{_{r,n}}$ with respect to these $n$-colorings is exactly as an edge in $\Kv{\matr{G}}_{_{\tilde{\varphi},\tilde{n}}}$, and consequently,
$\matr{G}_{_{\tilde{\varphi},n}}[\vect{x};\vect{y}]$ simulates $\tilde{\varphi}$ through $n$-colorings.\\
Hence, finally, by Lemma~\ref{lem:geninvfunc} we may define
$$\matr{G}_{_{\varphi,n}}[\vect{u};\vect{v};R] \isdef \matr{G}_{_{\tilde{\varphi},n}}[(\vect{u},\vect{s}_{_{0}});\vect{v}|_{0,\ldots,q};R].$$
Now, if for some $\vect{u}$, $\varphi(\vect{u})$ is supposed to be {\it undefined}, we may use the technique we adopted in Proposition~\ref{pro:edgesimulation}
to single out the value $*$ by a suitable simulator subgraph and make the corresponding coloring illegal.\\
The number of vertices and edges follows directly from the construction, Proposition~\ref{pro:permutations}  and  Proposition~\ref{pro:edgesimulation}.
}\end{proof}
\subsection{On the effectiveness of the function simulation }
In this section we consider the computational effectiveness of our constructions. To begin, let us discuss some computational aspects of the
codings we are going to use for graphs and functions. As a matter of fact, the coding one uses to feed the objects as {\it inputs} to an {\it algorithm} as a constructive solution of a problem may have a tremendous effect on the computational complexity of the solution itself
(for instance, consider a $\{0,1\}$-coding for the instances of a decision problem where the {\it yes} instances start with a $1$ and {\it no}
instances start with a $0$). Hence, it is quite important to fix a legitimate coding of objects before we start our computational analysis of the algorithms.\\
In this regard, let us start with the coding of functions. In what follows $\lfloor \varphi \rfloor_{_{i}}$ stands for the coding that $\{0,1\}$-encodes
the function $\varphi: A \to B$ when $A$ and $B$ are finite sets, as a set of ordered pairs. Then, it is clear that the size of this encoding
$|\lfloor \varphi \rfloor_{_{i}}|$ is of order $O(|A|(\log |A|+\log |B|))$. On the other hand, if we consider the (quantum) encoding of the same
function based on Lemma~\ref{lem:geninvfunc} consisting of $\Theta$ transpositions, then we denote this coding of the function $\varphi$ by
$\lfloor \varphi \rfloor_{_{Q}}$ and it is clear that its length, $|\lfloor \varphi \rfloor_{_{Q}}|$, is of order $O(\Theta \log r)$ where $r$
is the size of the domain for the invertible extension of $\varphi$ obtained from Lemma~\ref{lem:geninvfunc}.\\
For a given graph $\matr{G}=(V,E)$, $\langle \matr{G}\rangle_{_{i}}$ stands for the $\{0,1\}$-encoding of the adjacency matrix of $\matr{G}$
whose size $|\langle \matr{G}\rangle_{_{i}}|$ is of order $O(|V|^2)$. Moreover, let $\Gamma$ be the set of alphabets (symbols) that one needs
to write down the amalgam constructions of this article (e.g. contains $(,),+,0,1,$ $\ldots,9,\sum,[,],...$). Then it is clear that each such amalgam
construction presents an encoding $\langle \matr{G}\rangle_{_{Q}}$ of the corresponding graph on the alphabet $\Gamma$, coming from a function
$\varphi:S^{^p} \to S^{^q}$ with the coding $\lfloor \varphi \rfloor_{_{Q}}$, whose size $|\langle \matr{G}\rangle_{_{Q}}|$ is of order
$O(\Theta (p+q)\log m)$.\\
Now, consider the space
$${\bf Fin}_{_{S}} \isdef \{\varphi:S^{^p} \to S^{^q} \ \ | \ \ p \geq 1, q \geq 1 \}$$
that essentially contains all functions with finite domain and range. Our analysis shows that there is an embedding of ${\bf Fin}_{_{S}}$
into the space of simple graphs along with suitably chosen encodings $\lfloor \varphi \rfloor_{_{Q}}$ and $\langle \matr{G}\rangle_{_{Q}}$  in a way that
\begin{enumerate}
\item{The size $|\langle \matr{G}\rangle_{_{Q}}|$ is of order $O(|\lfloor \varphi \rfloor_{_{Q}}|)$.}
\item{The size $|\langle \matr{G} \rangle_{_{i}}|$ is of order $O(|\lfloor \varphi \rfloor_{_{i}}|^4)$.}
\item{Construction of $\langle \matr{G}\rangle_{_{Q}}$ given $\lfloor \varphi \rfloor_{_{Q}}$, (or $\langle \matr{G}\rangle_{_{i}}$ given $\lfloor \varphi \rfloor_{_{i}}$) is effective and can be obtained in
polynomial time. }
\end{enumerate}
We also have to consider the evaluation of a given function in ${\bf Fin}_{_{S}}$.
For this, first,  we go through some preliminaries.
Any marked graph
 $$\varrho:\{1,2,\ldots,|V(\matr{G})|\} \longrightarrow \matr{G}$$
 will be called an {\it ordered} graph and
will be denoted in an abbreviated style as $\matr{G}[\varrho]$ or
$\matr{G}[\nu]$ when $|V(\matr{G})|=\nu$ and
$\varrho$ is clear from the context (usually assumed to be $i \mapsto v_{_{i}}$).
Let $t$ be a fixed integer and ${\cal L}=\{L_{_{1}},L_{_{2}},\ldots,L_{_{\nu}}\}$ be a set of lists for
which $L_{_{i}} \subseteq \{1,2,\ldots,t\}$ for each $i$ and
$$\|{\cal L}\| \isdef \displaystyle{\sum_{i=1}^{\nu}}\ |L_{_{i}}|.$$
A graph $\matr{G}$ with $|V(\matr{G})|=\nu$ marked by ${\cal L}$ is
called a {\it list-graph} and is denoted by $\matr{G}[{\cal L}]$
(when we assume that the mapping is $L_{_{v}} \mapsto  v$).\\
Given a list-graph $\matr{G}[{\cal L}]$, the corresponding list coloring problem is the problem of finding a proper coloring
$\sigma: V(\matr{G}) \to \{1,2,\ldots,t\}$ in such a way that for any vertex $v$ we have $\sigma(v) \in L_{_{v}}$.
The {\it direct forcing rule} applied on an edge $e=uv \in E(\matr{G}[{\cal L}])$ is a rule that changes the list-graph $\matr{G}[{\cal L}]$
to a list-graph $\matr{G}[{\cal L}']$ such that
$$
L'_{_{z}} \isdef \left \{
   \begin{array}{ll}
    L_{_{z}}-L_{_{u}} & \text{if $z=v$ and $|L_{_{u}}|=1$,}  \\
    \emptyset &  \text{if $z=v$ and $L_{_{u}}=\emptyset$,}\\
    L_{_{z}} & \text{otherwise.}
   \end{array}\right.
$$
Observe that the list coloring problems corresponding to
$\matr{G}[{\cal L}]$ and $\matr{G}[{\cal L}']$ have the same
set of solutions.
The  direct forcing rule applied on an edge $e \in E(\matr{G}[{\cal L}])$ is denoted by ${\sf R}_{_{e}}$ and the list-graph obtained by applying
this rule is denoted by $\matr{G}[{\cal L}']=\matr{G}[{\cal L}]{\sf R}_{_{e}}$.\\
The {\it reduction} relation $\Rightarrow$ on list-graphs is defined as follows.
\begin{defin}{
We say that a list-graph $\matr{G}[{\cal L}]$ {\it is reduced to} $\matr{G}[{\cal L}']$ and we write
$\matr{G}[{\cal L}] \Rightarrow \matr{G}[{\cal L}'],$
if ${\cal L}' \not = {\cal L}$ and there exists an edge $e \in E(\matr{G}[{\cal L}])$ such that
$\matr{G}[{\cal L}']=\matr{G}[{\cal L}]{\sf R}_{_{e}}$.\\
As usual, $\Rightarrow^*$ stands for the reflexive and transitive closure of $\Rightarrow$.
 }\end{defin}
 \begin{lem}\label{lem:normalform}
 The reduction $\Rightarrow$ is terminating and confluent.
 \end{lem}
 \begin{proof}{%
 	First, we show that $\Rightarrow$ is terminating; that is,
	there is no infinite reduction sequence
	$$\matr{G}[{\cal L}_{_{1}}]\Rightarrow\matr{G}[{\cal L}_{_{2}}]\Rightarrow
	\matr{G}[{\cal L}_{_{3}}]\Rightarrow \cdots\;.$$
	This is because each reduction removes at least one color from one of the lists
	and the total size of the lists $\|{\cal L}_{_{1}}\|$ is finite and non-negative.\\
	Next, let us verify that $\Rightarrow$ is confluent.
	Since $\Rightarrow$ is terminating, by Newman's lemma (see e.g.~\cite{BANI99})
	it is enough to verify that $\Rightarrow$ is locally confluent; that is,
	$\matr{G}[{\cal L}]\Rightarrow\matr{G}[{\cal L}_{_{1}}]$ and
	$\matr{G}[{\cal L}]\Rightarrow\matr{G}[{\cal L}_{_{2}}]$ implies that
	there exist another list-graph $\matr{G}[{\cal L}']$ such that
	$\matr{G}[{\cal L}_{_{1}}]\Rightarrow^*\matr{G}[{\cal L}']$ and
	$\matr{G}[{\cal L}_{_{2}}]\Rightarrow^*\matr{G}[{\cal L}']$.\\
	The latter is equivalent to saying that for every two edges
	$e_{_{1}}$ and $e_{_{2}}$ in $\matr{G}$, there are two sequences of edges
	$a_{_{1}},a_{_{2}},\ldots,a_{_{k}}$ and $b_{_{1}},b_{_{2}},\ldots,b_{_{l}}$ such that
	${\sf R}_{e_{_{1}}}{\sf R}_{a_{_{1}}}{\sf R}_{a_{_{2}}}\cdots{\sf R}_{a_{_{k}}}=
	{\sf R}_{e_{_{2}}}{\sf R}_{b_{_{1}}}{\sf R}_{b_{_{2}}}\cdots{\sf R}_{b_{_{l}}}$.\\
	Now, if $e_{_{1}}$ and $e_{_{2}}$ are two arbitrary edges in $\matr{G}$,
	it is easy to verify that
	${\sf R}_{e_{_{1}}}{\sf R}_{e_{_{2}}}{\sf R}_{e_{_{1}}} =
	{\sf R}_{e_{_{2}}}{\sf R}_{e_{_{1}}}{\sf R}_{e_{_{2}}}$.
	It follows that the reduction relation~$\Rightarrow$
	is locally confluent.
 }\end{proof}
 \begin{defin}{As a direct consequence of Lemma~\ref{lem:normalform} one may deduce that every list-graph $\matr{G}[{\cal L}]$ has a unique normal form,
 denoted by $\matr{G}[{\cal L}^*]$, such that $\matr{G}[{\cal L}] \ \Rightarrow^* \ \matr{G}[{\cal L}^*]$ and $\matr{G}[{\cal L}^*]$ is irreducible
 (see e.g.~\cite{BANI99}).
 }\end{defin}
 In the following proposition $\langle \matr{G}[{\cal L}] \rangle_{_{i}}$ stands for a coding of the list-graph $\matr{G}[{\cal L}]$
 that is based on the $\{0,1\}$-encoding of its adjacency matrix along with a binary encoding of the contents of each list assigned to
 the vertices.
\begin{pro}\label{pro:polyalgorithm}
There exists a polynomial-time algorithm that receives the coding of a list-graph  $\langle \matr{G}[{\cal L}] \rangle_{_{i}}$ as the input
and produces the irreducible normal form $\langle \matr{G}[{\cal L}^*] \rangle_{_{i}}$ as the output.
\end{pro}
\begin{proof}{
Consider the algorithm that repeatedly scans through
the list-graph $\matr{G}[{\cal L}]$ (in a canonical predefined order)
and applies the direct forcing rule on every edge,
until the rule is not applicable on any edge anymore.
The list-graph obtained at the end is clearly
the normal form of $\matr{G}[{\cal L}]$.\\
To estimate the running time of the algorithm,
note that each round of this algorithm takes
$O(|\langle \matr{G}[{\cal L}] \rangle_{_{i}}|)$ steps.
Furthermore, the length of every reduction sequence
$$\matr{G}[{\cal L}_{_{1}}]\Rightarrow\matr{G}[{\cal L}_{_{2}}]\Rightarrow
	\matr{G}[{\cal L}_{_{3}}]\Rightarrow \cdots\Rightarrow\matr{G}[{\cal L}_{_{k}}]$$
is bounded by $\|{\cal L}\|$, and hence,
the algorithm performs at most $\|{\cal L}\|$ rounds before it stops.
Therefore, the algorithm finds the normal form of $\matr{G}[{\cal L}]$
in $O(\|{\cal L}\|\times|\langle \matr{G}[{\cal L}] \rangle_{_{i}}|)$
time steps.
}\end{proof}
Now, we may prove our main computability result.
\begin{thm}\label{thrm:effectivealg}
There is a polynomial-time algorithm that,
given a graph \linebreak $\matr{G}_{_{\varphi,n}}[X;Y;R]$
constructed in Theorem~\ref{thrm:main} and
an assignment $\sigma_{_{0}}: X  \cup R \to C$
with $\sigma_{_0}(\Kv{i})=i$ for $\Kv{i}\in R$,
finds the unique extension of $\sigma_{_0}$ to a
proper $C$-coloring of $\matr{G}$ whenever such a coloring exists
and declares when no such extension exists.
\end{thm}
\begin{proof}{Given $\sigma_{_0}$ define the list assignment
${\cal L}\isdef \{L_z: z\in V(\matr{G})\}$ with
$$
L_{_{z}} \isdef \left \{
   \begin{array}{ll}
     \{\sigma_{_0}(z)\}\;, &  \text{if $\sigma_{_0}(z)$ is defined,}\\
     C\;, & \text{otherwise.}
   \end{array}\right.
$$
Considering Lemma~\ref{lem:normalform} and Proposition~\ref{pro:polyalgorithm}, it is sufficient to verify that in the normal form $\matr{G}_{_{\varphi,n}}[{\cal L}^*]$ of $\matr{G}_{_{\varphi,n}}[{\cal L}]$, either all the lists in ${\cal L}$ are of size one and present the unique extension of $\sigma_{_0}$ to a proper $n$-coloring,
or all the lists in ${\cal L}$ are empty.\\
By  Lemma~\ref{lem:Lfunction} it is straight forward to verify that
if $\Lkn[{\cal L}^*]$ is the irreducible normal form of the list coloring problem $\Lkn[{\cal L}]$ with
$$
L_{_{z}} \isdef \left \{
   \begin{array}{ll}
     \{i\}\;, &  \text{if $z=\Kv{i}\in R$,}\\
     \{s\}\;, &  \text{if $z=x$,} \\
     C\;, & \text{otherwise,}
   \end{array}\right.
$$
for some $s \in S$, then
\begin{enumerate}
\item{If $s=k$ then $L_{_{y}}=\{k\}$.}
\item{If $s\not = k$ then $L_{_{y}} \subseteq C_{_{\{k\}}}$.}
\end{enumerate}
Using this one may verify the following claim.
\begin{itemize}
\item{{\bf Claim:}
{\it If $\Lkn[{\cal L}^*]$ is the irreducible normal form of the list coloring problem $\Lkn[{\cal L}]$ with
$|L_{_{x}}|=|L_{_{y}}|=|L_{_{\Kv{i}}}|=1 \ \ (i \in C)$, then all lists of ${\cal L}^*$ are of length one and they present a proper
$C$-coloring of $\Lkn$.
}}
\end{itemize}
Using this claim it is easy to verify the theorem for the graph that simulates a permutation $\pi$ and is obtained from Proposition~\ref{pro:permutations}. Now, by a careful inspection of the constructions we may also verify that
\begin{itemize}
\item{All our gadgets are constructed using the subgraphs $\eg[u,v]$, $\matr{G}_{_{\Cs^{+}_{_{m}},n}}$, $\matr{G}_{_{\Cs^{-}_{_{m}},n}}$ and $\Lkn$ as atoms.}
\item{There is no feedback in our constructions except through $R$ which is initially colored.}
\end{itemize}
Hence, the theorem is also valid for the generalized edge $\matr{E}_{_{r,n}}[\vect{u},\vect{v};\vect{v};R]$, and consequently is valid in general
for the whole simulation process.
}\end{proof}

\section{Appendix: graph amalgams}\label{APPNDX}
Following  \cite{DHT?}, and what we discussed in Section~\ref{INTRO}, note that if
$(X,\matr{G},\varrho)$ is a marked graph and
$\varsigma: X \longrightarrow
Y$ is a (not necessarily one-to-one) map, then one can obtain a new
marked graph $(Y,\matr{H},\tau)$ by considering the push-out of the
diagram
$$\matr{Y} \stackrel {\varsigma}{\longleftarrow} \matr{X}
\stackrel {\varrho}{\longrightarrow} \matr{G}$$ in the category of
graphs. It is easy to check that the push-out exists and is a
monomorphism. Also, it is easy to see that the new marked graph
$(Y,\matr{H},\tau)$ can be obtained from $(X,\matr{G},\varrho)$ by
identifying the vertices in each inverse-image of $\varsigma$.
Hence, again we may denote
$(Y,\matr{H},\tau)$ as
$\matr{G}[\varsigma(x_{_{1}}),\varsigma(x_{_{2}}),\ldots,\varsigma(x_{_{k}})]$
where we allow repetition in the list appearing in the brackets.
Note that with this notation one may interpret $x_{_{i}}$'s as a set
of {\it variables} in the  {\it graph structure}
$\matr{G}[x_{_{1}},x_{_{2}},\ldots,x_{_{k}}]$, such that when one
assigns other (new and not necessarily distinct) {\it values} to
these variables one can obtain
some other graphs (by identification of vertices).\\
On the other hand, given two marked graphs $(X,\matr{G},\varrho)$
and $(Y,\matr{H},\tau)$ with
$X=\{x_{_{1}},x_{_{2}},\ldots,x_{_{k}}\}$ and
$Y=\{y_{_{1}},y_{_{2}},\ldots,y_{_{l}}\}$, one can construct their
amalgam $(X,\matr{G},\varrho)+(Y,\matr{H},\tau)$ by forming the
push-out of the following diagram
$$\matr{H}  \stackrel {\tilde{\tau}}{\longleftarrow}  \matr{X} \cap \matr{Y}
\stackrel {\tilde{\varrho}}{\longrightarrow} \matr{G},$$ in which
$\tilde{\tau} \isdef \tau|_{_{X \cap Y}}$ and $\tilde{\varrho}
\isdef  \varrho|_{_{X \cap Y}}$. Following our previous notations we
may denote the new structure by
$$\matr{G}[x_{_{1}},x_{_{2}},\ldots,x_{_{k}}]+\matr{H}[y_{_{1}},y_{_{2}},\ldots,y_{_{l}}]$$
if there is no confusion about the definition of mappings. Note that
when $\matr{X} \cap \matr{Y}$ is the empty set, then the amalgam is
the {\it disjoint union} of the two marked graphs. Also, by the
universal property of the push-out diagram, the amalgam can be
considered as marked graphs marked by $X$, $Y$, $X \cup Y$ or $X
\cap Y$.



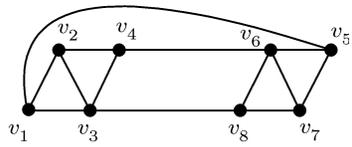
\end{document}